\documentclass[10pt,a4paper]{amsart}
          
\usepackage{mystyle}
\usepackage{tikz-cd}

\title{Some positivity results for toric vector bundles}
\author{Bernt Ivar Utst$\o$l N$\o$dland}

\begin{document} 
\maketitle

\begin{abstract}
We give a criterion for a projectivized toric vector bundle to be a Mori dream space and describe its Cox ring using generators and relations. Both  of these results are in terms of the matroids of all symmetric powers of the bundle. We also give a criterion for a toric vector bundle to be big and describe several interesting examples of toric vector bundles which highlights how positivity properties for toric vector bundles are more complicated than for toric line bundles.
\end{abstract}

\section{Introduction}

Positivity is an important notion in the study of the geometry of projective varieties. For toric varieties we have a very good understanding of various positivity properties of line bundles, in terms of the combinatorics defining the variety. In this paper we investigate various positivity properties of toric vector bundles. Equivalently, we study positivity of line bundles on projectivized toric vector bundles.

In \cite{DJS}, Di Rocco, Jabbusch and Smith associate to a toric vector bundle $\E$ a representable matroid $M(\E)$. To each element $e$ in the ground set of $M(\E)$, there is an associated divisor $D_e$, such that $\oplus_e \OO(D_e)$ surjects onto $\E$, and the induced map  on global sections is surjective.

In this paper we investigate the connection between the matroid and positivity properties of $\E$. Positivity of line bundles is closely related to sections of multiples of the bundle. The corresponding  notion for vector bundles is symmetric powers. An important property of the matroid $M(\E)$ is that it does not necessarily commute with taking symmetric powers when the rank is at least $3$. This makes the study of vector bundles significantly harder than the study of line bundles. To be able to study positivity,  we thus have to study the matroids of all symmetric powers $S^k \E$ at the same time. To do this we  define a set of vectors $\mathfrak{M}(\E)$ containing the ground set of the matroid $M(\E)$, but also containing all matroid vectors in the ground set of some  symmetric power $S^k \E$ which cannot be written as a symmetric product of matroid vectors for lower symmetric powers. The first main result in this paper is the following:

\begin{theorem} [\cref{theorem:criterion}]
Let $\E$ be a toric vector bundle on the smooth toric variety $X_\Sigma$. Then $\Cox(\p(\E))$ is finitely generated if and only if $\mathfrak{M}(\E)$ is finite.
\end{theorem}

The second main result is a presentation of the Cox ring of $\p(\E))$ in terms of generators and relations. The generators correspond bijectively to the set $\mathfrak{M}(\E) \cup \Sigma(1)$. We are also able to describe all relations, in terms of  relations between vectors in $\mathfrak{M}(\E)$, see \cref{posSection:coxPresentation} and in particular \cref{theorem:coxpresenetation}.

Our results on Cox rings extends and reproves many of the results in \cite{GHPS}, using different techniques: A key point in the paper is to use the Klyachko filtrations directly.  Our results are also more general. A heuristic explanation for our results on Cox rings is the following: We have that $M(\E \otimes \Li) = M(\E)$, for any line bundle $\Li$, since tensoring with a $1$-dimensional vector space does not change  linear algebra relations. Thus the matroid of $\E$ is most naturally considered an invariant of $\E \otimes \Li$, where $\Li$ is allowed to vary freely. To study the Cox ring we need to study sections of $S^k \E \otimes \Li$, where $k$ and $\Li$ are allowed to vary. However it is sufficient to restrict to algebra generators of the Cox ring over $\Cox(X_\Sigma)$ and these are exactly sections corresponding to the set $\mathfrak{M}(\E)$.

Our third main result is a criterion for a toric vector bundle to be big. 

\begin{theorem} [\cref{theorem:bignessCriteria}]
A toric vector bundle is big if and only if there exists $k>0$ and $v \in M(S^k \E)$, such that the associated polytope is full dimensional.
\end{theorem}

Additionally we give some other interesting examples and results which are related to positivity of toric vector bundles:
\begin{itemize}
    \item A big toric vector bundle with the property that no Minkowski sum of the polytopes in the parliament is full-dimensional (\cref{example:bignominkowski}).
        \item A way of interpreting the nefness/ampleness of a toric vector bundle in terms of a notion of concavity  for the  piecewise linear support function on the associated branched cover of a fan (\cref{proposition:branchedcovercriterion}).
    \item A toric surface with ample rank two vector bundles $\E_k$ such that $S^k \E_k$ is not globally generated (\cref{counterexample}).
    \item A sequence of toric vector bundles showing that there cannot exist a bound depending on the dimension  of the variety and/or the rank of the vector bundle with the following property: If $\E$ is ample and the degree of $\E|_C$ is larger than this bound for any invariant curve $C$, then $\E$ is globally generated or very ample (\cref{example:noggbound}).

\end{itemize}

{\bf Acknowledgements.}
I am grateful to John Christian Ottem for numerous helpful conversations on the topics in this paper, as well as comments on earlier versions of this paper. I also wish to thank Milena Hering, Nathan Ilten and Ragni Piene for comments on an earlier version of this paper.

\section{Preliminaries on toric varieties} \label{posSection:prelim}

We here recall some preliminary material on toric varieties and toric vector bundles, most of this can be found in \cite{Fulton} and \cite{CLS}.  Let $T= (\C^\ast)^n$ be an algebraic torus and denote by $M$ its character lattice $\Hom(T,\C^\ast)$ and by $N$ its dual lattice of one-parameter subgroups. A toric variety  $X$ will in this paper denote a normal irreducible variety containing $T$ as an open dense subset, such that the action of $T$ on itself extends to an action on $X$. It is well known that any toric variety corresponds to a fan $\Sigma$ in $N_\Q = N \otimes \Q$. We will denote the toric variety associated to $\Sigma$ by $X_\Sigma$. In this paper we assume that the ground field equals $\C$ and that $X_\Sigma$ is smooth and complete. This is because the theory of parliaments of polytopes is only developed under these assumptions. We believe that much of the following will remain true with only minor modifications in the case of any $\Q$-factorial toric variety, in other words for any simplicial fan.

Any divisor $D$ on $X_\Sigma$ is linearly equivalent to a sum $D=\sum_\rho a_\rho D_\rho$ of $T$-invariant divisors. Alternatively it corresponds to a piecewise linear support function $\phi_D : N_\Q \to \Q$ given by 
\[ x \mapsto \langle m_\sigma,x \rangle, \]
where $\sigma$ is any cone containing $x$ and $m_\sigma$ is Cartier data for $D$ on $U_\sigma$. In other words $m_\sigma$ satisfies $a_\rho = \langle m_\sigma,\rho \rangle$ for any ray $\rho$ of $\sigma$.

Associated to a divisor $D$ there is a polytope $P_D$ defined by
\[ P_D = \{ x \in M_\Q | \langle x,\rho \rangle \leq a_\rho \}. \]
We have the following well-known formula for the global sections of $D$ \cite[p. 66]{Fulton}:
\[ H^0(X_\Sigma,\OO(D)) \simeq \bigoplus_{m \in P_D \cap M} \C \chi^m.\] 

\begin{remark} 
In the above we have used the convention that the Cartier data $m_\sigma$ of a divisor $D = \sum a_\rho D_\rho$ satisfies $\langle m_\sigma,\rho \rangle = a_\rho$, while many texts on toric geometry use the convention that it satisfies $\langle m_\sigma,\rho \rangle = -a_\rho$.  This has the consequence that polytopes will be drawn in the opposite direction of what is usual in toric geometry, for instance in \cite{CLS} and \cite{Fulton}.  This is done because we are largely interested in the parliament of polytopes studied in \cite{DJS}, which uses this convention. This has the consequence that some formulas and results will have an extra minus-sign compared to their usual formulations. For a similar reason we will talk about concave support functions, instead of the usual convex ones.
\end{remark}

A toric vector bundle $\E$ is a vector bundle on a toric variety $X_\Sigma$ together with a $T$-action on the total space of the vector bundle, making the bundle projection $\E \to X$ into a $T$-equivariant morphism, such that for any $t \in T, x \in X_\Sigma$ the map $\E_x \to \E_{t \cdot x}$ is linear. 
The study of toric vector bundles goes back to Kaneyama \cite{Kaneyama} and Klyachko \cite{Klyachko}, who both gave classifications of toric vector bundles in terms of combinatorial and linear algebra data. Klyachko applied this to study, among other things, splitting of low rank vector bundles on $\p^n$. We here recollect Klyachko's description:

To a toric vector bundle $\E$ of rank $r$ on $X_\Sigma$, we let $E \simeq \C^r$ denote the fiber at the identity of the torus. Klyachko shows that there  for each ray $\rho \in \Sigma(1)$ is an associated filtration $E^\rho (j)$ of $E$, indexed over $j \in \Z$. It has the property that for any ray $E^\rho(j) = 0$ for $j$ sufficiently large and $E^\rho(j) = E$ for $j$ sufficiently small. Additionally it satisfies a compatibility condition:

For any maximal cone $\sigma$ there exists characters $u_1,...,u_r \in M$ and vectors $L_u \in E$ such that for any ray $\rho$ of $\sigma$ we have
\[ E^\rho(j)  = \sum_{j | \lb u,\rho \rb \geq j} L_u.\]

The above decomposition is equivalent to the fact that on the affine $U_\sigma$, $\E$ splits into a direct sum of line bundles $\OO(u_i)$. Klyachko's classification theorem is the following:

\begin{theorem} [{\cite[Thm 0.1.1]{Klyachko}}]
The category of toric vector bundles on $X$ is equivalent to the category of finite dimensional vector spaces $E$, with filtrations indexed by the rays as described above, satisfying the compatibility condition. A morphism $\E \to \F$ corresponds to a linear map $E \to F$, respecting the filtrations.
\end{theorem}

\section{Parliaments of polytopes} \label{posSection:parliament}

 We next briefly recall the notion of a parliament of polytopes introduced in \cite{DJS}, which is a way of describing global sections of $\E$ in terms of lattice points in a collection of polytopes.

For a toric vector bundle the cohomology groups $H^i(X_\Sigma,\E)$ decompose as a direct sum $\oplus_{u \in M} H^i(X_\Sigma,\E)_u$, over the $\chi^u$-isotypical components. Klyachko showed that 
\[ H^0(X_\Sigma,\E)_u = \cap_{\rho \in \Sigma(1)} E^\rho(\langle u,\rho \rangle ). \]
The notion of parliaments of polytopes gives a more detailed way of studying $H^0(X_\Sigma,\E)$.

Consider the set of all intersections of the form $\cap_\rho E^{\rho}(j_\rho)$. There is a unique representable matroid $M(\E)$ associated to $\E$, whose ground set is constructed inductively as follows: For any intersection of dimension one, add a vector from this intersection to the ground set. Assume that we have added vectors corresponding to all intersections of dimension $i$ and let $V$ be an intersection of dimension $i+1$. Let $G$ be the set of all such vectors contained in $V$. Let $W$ be a complementary subspace to $\Span(G)$ inside $V$. We choose a basis for $W$ and add all the basis vectors to the ground set of the matroid. Performing this process on all possible such intersections, we get a set of vectors whose associated matroid $M(\E)$ is uniquely determined by $\E$. We will often, by abuse of notation, identify the matroid with a fixed choice for its ground set.

By construction $M(\E)$ has the property that each of the intersections $\cap_\rho^n E^{\rho}(j_\rho)$ can be written as the span of vectors from a ground set of the matroid. To each vector $e \in M(\E)$  we associate a divisor $D_e = \sum_\rho a_\rho D_\rho$, where 
\[ a_\rho = \mathrm{max} \{j \in \Z: e \in E^{\rho}(j)\}. \]
We denote by $P_e$ the associated polytope:
\[P_e = \{ u \in M_\R | \lb u,\rho \rb \leq \mathrm{max} (j \in \Z: e \in E^{\rho}(j)), \text{ for all } \rho \}. \]

\begin{proposition} [{\cite[Proposition 1.1]{DJS}}]
The lattice points in the parliament of polytopes correspond to a torus equivariant generating set of $H^0(X,\E)$.
\end{proposition}

A crucial difference from the case of line bundles on toric varieties is that the sections associated to all the lattice points aren't necessarily linearly independent. If the polytopes are disjoint then then they actually give a basis, but if there is any overlap there might be relations among them. The matroid structure of the indexing set describes precisely the dimension of a $\chi^u$-isotypical component.
\begin {proposition} \label{proposition:dimglobalchru}
Given a toric vector bundle $\E$ and a character $u \in M$ we have
\[ \dim H^0(X_\Sigma,\E)_u = \dim \mathrm{Span} \{e \in M(\E) | u \in P_e \}. \]
\end{proposition}
\begin{proof}
This follows by the equivalence from \cite[proof of Proposition 1.1]{DJS}:
\[ e \in H^0(X_\Sigma,\E)_u= \cap E^\rho(\lb u,\rho \rb) \Leftrightarrow u \in P_e \cap M. \]
\end{proof}

An equivalent way of formulating the above statements on global sections is to observe that by construction there is a surjection
\[ \F=\bigoplus_{e \in M(\E)} \OO(D_e) \to \E  \to 0, \]
which is surjective on global sections \cite[Remark 3.6]{DJS}. Indeed, due to Klyachko's equivalence of categories such a map corresponds to a surjective map of vector spaces $\psi: F \to E$. The vector space $F$ has a basis consisting of one basis vector $w_e$ for each $e \in M(\E)$. The map $\psi$ is simply the map sending $w_e$ to $e$. By construction it will be surjective on global sections.

We think the following observation will be useful for studying parliaments of polytopes: Observe that the above surjection corresponds to a closed embedding
\[ i:\p(\E) \to  \p(\F) \]
such that $i^\ast \OOF = \OOO$. $\p(\F)$ is itself a toric variety $X_{\Sigma'}$ whose fan lives in $N_\Q' = N_\Q \oplus N_Q''$, for a lattice $N''$ of rank equal to the number of matroid vectors minus one. Thus $\OOF$ defines a polytope $P_{\OOF}$ in $M_\Q'$ whose lattice points corresponds to elements of $H^0(X_{\Sigma'},\OOF) = H^0(X_\Sigma,\F)$. Moreover the parliament for $\F$ also corresponds to global sections of $\F$.

We recall the construction of the fan $\Sigma'$ in $N'_\Q = N_\Q \oplus N_\Q''$. Let $D_0,\ldots,D_s$ be the divisors in the parliament of $\E$ and write 
\[ D_i = \sum_\rho a_{i \rho} D_\rho \]
There are rays $w_0,\ldots,w_s$ in $N_\Q''$ having the structure of the fan of $\p^s$, exhibiting the $\p^s$-bundle structure of $\p(\F)$. For any ray $\rho \in \Sigma(1)$ we have an associated ray $\rho'$ of $\Sigma'$ with minimal generator
\[ \rho'= \rho + \sum_i a_{i \rho} w_i. \]
A cone of $\Sigma'$ is the Minkowski sum of any cone of $\Sigma$ plus any cone generated by  a proper subset of $\{w_0,\ldots,w_s \}$.

\begin{lemma} \label{lemma:bigpoly}
For fixed $i$, the rational equivalence class of the divisor  
\[  D_{w_i} + \sum_\rho a_{i \rho} D_{\rho'},\]
equal that of $\OOF$. Fixing any one such representation, the polytopes $P_e$ are obtained as the intersections of $P_{\OOF}$ with the fibers over the $s+1$ lattice points of the standard simplex $\Delta_s$ in $M_\Q''$, along the map $M_\Q \to M_\Q''$.
\end{lemma}
\begin{proof}
Set $\G = \F \otimes \OO(-D_{i})$.  Both $X_\Sigma$ and $\p(\G)$ are toric varieties, thus their Picard groups are easily computable from the combinatorics of the fans. Since $\p(\G)$ is a projective bundle over $X_\Sigma$, we also know the relationship between these Picard groups. Comparing the two ways of computing these groups, we see that the class of $D_{w_i}$ has to equal the class of $\OO_{\p(\G)}(1)$, or its negative. But $D_{w_i}$ is effective, thus it has to be the positive $\OO_{\p(\G)}(1)$.  But this implies that on $\p(\F)$ the class of $\OOF$ equals that of $D_{w_i} + D_i$.

We now fix a representation of $\OOF$, say $\OOF = D_{w_0} + \sum_\rho a_{0 \rho} D_{\rho'}$. The polytope $P_{\OOF}$ is given by the inequalities, for some $(x,y) \in N_\Q \otimes N_\Q''$
\[ x_1 \leq 0 \]
\[  \vdots  \]
\[ x_s \leq 0 \]
\[ -x_1-\cdots-x_s \leq 1 \]
\[ \langle (x,y),\rho' \rangle \leq a_{0 \rho} \]
Let $F_0$ be the subset of $P_{\OOF}$ where $x_1=\ldots=x_s=0$. Then we see that the above inequalities reduce to
\[ \langle y,\rho \rangle \leq a_{0 \rho} \]
Thus $F_0$ is exactly $P_{D_0} \times (0,\ldots,0) \subset N_\Q \otimes N_\Q''$. Similarly, if $F_i$ is the locus with $x_i=-1$ and $x_j=0$ for $j \neq i$, then $F_j$ is the polytope given by
\[ \langle y,\rho \rangle + \langle (0,\ldots,-1,\ldots,0), \sum_j a_{j \rho} w_j \rangle \leq a_{0 \rho} \]
which after cancelling $a_{0 \rho}$ is given by
\[ \langle y,\rho \rangle \leq a_{i \rho} \]
Thus $F_i$ is $P_{D_{w_i}}$ times a point.
\end{proof}

\begin{corollary}
The parliament of polytopes of $\E$ is obtained by projecting the $s+1$ polytopes $M \times v_i \cap P_{\OOF}$ (which could be empty), where the $v_i$ are the vertices of the standard $s$-simplex $\Delta_s$ in $M_\Q''$. Moreover the lattice points in the parliament of polytopes are the images of all lattice points of $P_{\OOF}$. 
\end{corollary}
\begin{proof}
Most of this is clear from the above lemma. Note that the lattice points in $P_{\OOF}$ are all $m' \in M'$ such that $H^0(X_\Sigma',\OOF)_{m'}$ is non-zero. By the inequalities defining the polytope we see that for $m'=(m,m'')$ we must have that $m''$ lies in the standard simplex in $M_\Q''$. But this has only $s+1$ lattice points (corresponding exactly to the polytopes  $F_i$ defined in the proof of \cref{lemma:bigpoly}), hence any corresponding $m$ lies in some face $F_i$.
\end{proof}

\begin{corollary}
The global sections of $\E$ are obtained by projecting $s+1$ polytopes along $p: M_\Q' \to M_\Q$ and then identifying sections according to the dependence structure of the matroid $M(\E)$.
\end{corollary}

\begin{remark}
When all polytopes in the parliament are non-empty, we have that the big polytope $P_{\OOF}$ is the Cayley polytope of the polytopes in the parliament. In that case the above statements follow from well-known results on Cayley polytopes. In that case we also have that under the projection $q: M_\Q' \to M_\Q''$ $P_{\OOF}$ is mapped to the standard simplex.  See for instance \cite[Section 2]{CayleyP} for details on Cayley polytopes.
\end{remark}

\begin{remark}
Fix $\E$, and define the polytope $Q$ as the convex hull of all polytopes in the parliament. If $\E$ is ample, then the fiber of $p$ of any point in the interior of $Q$ intersected with $P_{\OOF}$ is non-empty, although it need not contain lattice points. Studying the size of the fiber, using a similar construction for complexity one $T$-varieties, under taking multiples of a line bundle $\Li$, is the idea used by Altmann and Ilten to prove Fujita's freeness conjecture for complexity one $T$-varieties \cite{AI}. In  particular it also is true for rank two toric vector bundles. Thus, for any ample line bundle $\Li$ on $X=\p(\E)$, where $\E$ has rank $2$, we have that $(\dim X +1) \Li + K_X$ is basepoint free. One could hope that it would be possible to use this construction to study also higher rank bundles, although the fact that the matroid of the symmetric power  $S^k \E$ does not equal the symmetric power of the ground set of $M(\E)$ makes it significantly harder. An example of a bundle for which the  matroid of the symmetric power does not equal the symmetric power of the matroid can be seen in \cref{example:big}.
\end{remark}

\begin{remark}
Di Rocco, Jabbusch and Smith ask \cite[p.3]{DJS} whether there is, for a globally generated toric vector bundle $\E$, a relation between regular triangulations of the parliaments of $\E$ and the equations of $\p(\E)$ under the embedding by the linear system $\OOO$. The above makes this plausible: The linear system $\OOO$ gives a rational map $\p(\F) \dashrightarrow \p^N = \p(H^0(X_{\Sigma'},\F))$. The surjection $H^0(X_\Sigma,\F) \to H^0(X_\Sigma,\E)$ implies that $\p(H^0(X_\Sigma,\E))$ is a linear subspace of $\p^N$. Then a regular triangulation of $P_{\OOF}$ induces a regular triangulation of the parliament of $\E$, under the projection $p$.
\end{remark}

\iffalse
\begin{proposition}
If $\E$ is an ample toric vector bundle of rank two then $\OO_{\p(\E)}(k) \otimes \det \E$ is globally generated for any $k \geq 0$.
\end{proposition}

\begin{proof}

\end{proof}

\begin{proposition}
NOTE TO SELF: Intersection numbers on $X_\Sigma'$ is given by $\OOF \cdot C' = 1$ when $C'$ is curve on $\p^s$ and $\OOF \cdot C' = D_i \cdot C$ when $C'$ corresponds to a  curve $C$ on $X_\Sigma$ and divisor $D_i$.
\end{proposition}

\fi

\section{Cox rings of projectivized toric vector bundles} \label{posSection:coxRings}

Let $X_\Sigma$ be a smooth projective toric variety. Let $\E$ be a toric vector bundle on $X_\Sigma$; we denote the natural map from $\p(\E)$ to $X_\Sigma$ by $\pi$. We wish to study the Cox ring of $\p(\E)$. Let $ \rho_1,\ldots,\rho_n$ be the rays of $\Sigma$, and denote by $D_i$ the torus-invariant divisor associated to $\rho_i$. By the description of the Picard group of a projective bundle we have
\[ \Cox(\p(\E)) \simeq \bigoplus_{k,k_1,\ldots,k_n \in \Z} H^0(\p(\E),\OO(k) +  \pi^\ast k_1D_1 +\ldots +  \pi^\ast k_n D_n ) \]
\[\simeq \bigoplus_{k,k_1,\ldots,k_n \in \Z} H^0(X_\Sigma,S^k \E \otimes k_1D_1 \otimes \cdots \otimes k_n D_n ). \]
We will study the latter $\C$-algebra.

By construction the matroid $M(\E)$ has the property that each of the intersections $\cap_\rho E^{\rho}(j_\rho)$ can be written as the span of vectors from the ground set of the matroid. We denote by $L(\E)$ the set of all such intersections. Recall that we constructed the divisor  $D_e$ associated to a matroid vector $e \in M(\E)$, given by  $D_e = \sum_\rho a_\rho D_\rho$, where 
\[ a_\rho = \mathrm{max} \{j \in \Z: e \in E^{\rho}(j)\}. \]
We will now need a slight generalization of these divisors; to each linear space  $V \in L(\E)$  we associate a divisor $D_V = \sum_\rho a_\rho D_\rho$, where 
\[ a_\rho = \mathrm{max} \{j \in \Z: V \subset E^{\rho}(j)\}. \]

We denote by $P_V$ the associated polytope. If $V$ is one-dimensional we often identify it with a non-zero vector $e$ in its span and identify the associated divisor (and polytope) with $D_e$ (and $P_e$).  The parliament of polytopes is the collection of polytopes $P_e$, for $e$ in the ground set of the matroid $M(\E)$.

In the following paragraphs we present several technical results on these divisors. The main motivation behind these is to study the multiplication maps
\[H^0(\p(\E),\Li_1) \otimes H^0(\p(\E),\Li_2) \to H^0(\p(\E),\Li_1 \otimes \Li_2) . \]
We show that any such multiplication map can be lifted to a multiplication map of sections of line bundles on the toric variety $X_\Sigma$. Using this, we give a criterion for $\p(\E)$ to be a Mori Dream Space.

Many of the conclusions we obtain on Cox rings were already shown in the paper \cite{GHPS}, using different methods. However we think that our perspective is illuminating for understanding toric vector bundles, since our results are formulated and proved using the Klyachko filtrations directly. Moreover, we are able to isolate precisely what the generators of the Cox ring of $\p(\E)$ are: They are pullbacks of sections from the base, together with sections corresponding to matroid vectors of some $S^k \E$ which are not symmetric powers of matroid vectors of symmetric powers of $\E$ smaller than $k$. A key ingredient for doing this is knowing the Klyachko filtrations of tensor products and symmetric powers of $\E$, which was described in \cite{GonzalezOk}:

\begin{proposition} [{\cite[Corollary 3.2]{GonzalezOk}}]  \label{prop:filtensorprod}
Let $\E_1,... \E_s$ be toric vector bundles  on $X_\Sigma$. Then the Klyachko filtrations of their tensor product $\E_1 \otimes \cdots \otimes \E_s$ are given by
\[ (E_1 \otimes \cdots \otimes E_s)^\rho(j) = \sum_{j_1+ \cdots + j_s=j} E_1^\rho(j_1) \otimes \cdots \otimes E_s^\rho(j_s), \]
for any ray $\rho \in \Sigma$ and $j \in \Z$.
\end{proposition}

\begin{proposition} [{\cite[Corollary 3.5]{GonzalezOk}}] \label{prop:filsympower}
Let $\E$ be toric vector bundles  on $X_\Sigma$. Then the Klyachko filtrations of its symmetric power $S^k \E$  are given by
\[ (S^k E)^\rho(j) = \sum_{j_1+ \cdots + j_k=j} \mathrm{Im} (E^\rho(j_1) \otimes \cdots \otimes E^\rho(j_k) \to S^k E), \]
for any ray $\rho \in \Sigma$ and $j \in \Z$.
\end{proposition}

We now proceed to use these descriptions to study what happens to the divisors in the parliament of polytopes under taking tensor and symmetric products.

\begin{lemma} \label{prop:tensorDivisor}
Given two vector bundles $\E$ and $\F$ and vector subspaces $V \subset E$ and $W \subset F$, we have that
\[ D_V + D_W = D_{V \otimes W}, \]
where $V \otimes W$ is regarded as a vector subspace in $E \otimes F$, the fiber over $\E \otimes \F$ at the identity.
\end{lemma}
\begin{proof}
The filtrations of $\E \otimes \F$ is given by the tensor product of the filtrations of  $E$ and $F$. The claim is local, so in other words we can check the coefficient of each ray separately. Restricted to a fixed ray we can write the filtration of $\E \otimes \F$ in terms of a basis for $E$ containing a basis for $V$ and a basis for $F$ containing a basis for $W$. Then we see that
\[c_\rho = \max \{ l \;| \; V \otimes W \subset (E \otimes F)^\rho(l) \}, \]
is exactly the sum $a_\rho+b_\rho$, where
\[a_\rho = \max \{ l \;|\; V  \subset E^\rho(l) \}, \]
\[b_\rho = \max \{ l \;|\; W \subset F^\rho(l) \}, \]
 proving the claim.
\end{proof}

\begin{corollary} \label{cor:divisortensorprod}
If $\E_1,\ldots,\E_s$ are vector bundles and $V_i \subset E_i$ are subspaces, then
\[ D_{V_1}+\ldots+D_{V_s} = D_{V_1 \otimes \cdots \otimes V_s}. \]
\end{corollary}
\begin{proof}
By induction on $s$.
\end{proof}

\begin{lemma}
Given a vector bundle $\E$ and subspaces $V_i \subset E$,  then 
\[ D_{V_1} + \ldots + D_{V_s} = D_{V_1 \cdots V_s}, \]
where $V_1 \cdots V_s$ is the subspace of $S^s E$ which by definition is the image of $V_1 \otimes \cdots \otimes V_s$ under the natural map $E^{\otimes s} \to S^s E$.
\end{lemma}
\begin{proof}
The filtrations of $S^s \E$ is given by taking sums of symmetric products of the subspaces appearing in the filtrations \cref{prop:filsympower}. The claim is local, so in other words we can check the coefficient of each ray separately. Writing $D_{V_1 \cdots V_s} = \sum_\rho c_\rho D_\rho$ and $D_{V_i} = \sum_\rho a_\rho^i D_\rho$ we have that
\[c_\rho = \max \{ l \;| \; V_1 \cdots V_s  \subset (S^s E)^\rho(l) \}. \]
By \cref{prop:filsympower} we have  that $c_\rho$ has to be greater than or equal to the sum $\sum_i a_\rho^i$. If it is actually greater, then we contradict \cref{cor:divisortensorprod}: there has to be some $v \in \cap (E^{\otimes s})^\rho(c_\rho)$ which is not in the vector space $W_0=\cap_\rho(E^{\otimes s})^\rho(\sum_i a_\rho^i)$. We can map $v$ to its image $\overline{v}$ in $S^s E$ and on to $w_1 \in W_1=\cap_\rho(E^s)^\rho(\sum_i a_\rho^i)$. Since $W_0 \to W_1$ is simply induced from the quotient morphism $E^{\otimes s} \to S^s E$ we can lift $w_1$ to some $w_0$ in $W_0$. 

By \cref{prop:filtensorprod} we observe that the vector spaces appearing in the filtrations of $\E^{\otimes s}$ are invariant under the action of the symmetric group on $s$ letters. Thus if $w_0$ is in $W_0$, then $v$ also has to be there, which is a contradiction.
\end{proof}

Recall that for a character $u \in M$, $H^0(X_\Sigma,\E)_u = \cap E^\rho(\lb u,\rho \rb)$. 

\begin{proposition}
	The natural map $H^0(\E)_u \otimes H^0(\F)_v \to H^0(\E \otimes \F)_{u+v}$ is given by sending $s \in \cap E^\rho(\lb u,\rho \rb)$ and $t \in \cap F^\rho(\lb v,\rho \rb)$ to $s \otimes t \in \cap (E \otimes F)^\rho(\lb u+v,\rho \rb)$.
\end{proposition}
\begin{proof}
We consider the local situation: On any open affine $U_\sigma$, where $\sigma \in \Sigma$, we have that the bundles split as a sum of line bundles, i.e.
\[ \E|_{U_\sigma} \simeq \oplus_{i=1}^s \OO(u_i), \]
\[ \F|_{U_\sigma} \simeq \oplus_{j=1}^t \OO(v_j). \]
Thus the multiplication map corresponds to a map
\[ \oplus_{i=1}^s H^0(\OO(u_i)) \otimes \oplus_{j=1}^t H^0(\OO(v_j))  \to \oplus_{i,j=1}^{s,t} H^0(\OO(u_i+v_j)), \]
which is given simply by sending $\chi^{u_i} \in H^0(\OO(u_i),U_\sigma), \chi^{v_j} \in H^0(\OO(v_j),U_\sigma)$ to $\chi^{u_i+v_j} \in H^0(\OO(u_i+v_j),U_\sigma)$.

Now by construction of the matroids, $u_i$ corresponds to some $e_i \in M(\E)$ and $v_j$ corresponds to some $f_j \in M(\F)$, thus the above  can be written as
\[ H^0(U_\sigma,D_{e_i}) \otimes H^0(U_\sigma,D_{f_j}) \to H^0(U_\sigma,D_{e_i}+D_{f_j}) = H^0(U_\sigma,D_{e_i \otimes f_j}). \]
Thus locally the multiplication of sections $s,t$ is given by $s \otimes t$, and the result follows.
\end{proof}

\begin{proposition}
Given a map $\E \to \F$ of vector bundles, corresponding to the linear map $\phi: E \to F$,  the induced map $H^0(\E) \to H^0(\F)$ is given by sending $s \in H^0(\E)_u= \cap E^\rho(\lb u,\rho \rb)$ to $\phi(s) \in H^0(\F)_u$.
\end{proposition}
\begin{proof}
This follows from Klyachko's classification theorem.
\end{proof}

\begin{corollary}
The natural map $S^a \E \otimes S^b \E \to S^{a+b} \E$ induces the map $H^0(S^a \E \otimes S^b \E) \to H^0(S^{a+b} \E)$ given by $s \otimes t \mapsto st$.
\end{corollary}

By the above we also get
\begin{proposition}
The natural map $H^0(S^a \E) \otimes H^0(S^b \E) \to H^0(S^{a+b} \E)$ is given by $s \otimes t \mapsto st$.
\end{proposition}
\begin{proof}
This follows from the above, since it is the composition
\[ H^0(S^a \E) \otimes H^0(S^b \E)  \to H^0(S^a \E \otimes S^b \E) \to H^0(S^{a+b} \E).\]
\end{proof}

\begin{lemma}  \label{lemma:matroidtensorl}
If $\E$ is a vector bundle and $\LL$ a line bundle then $M(\E \otimes \LL) \simeq M(\E)$ under any isomorphism $E \otimes L \simeq E$. If $v \in M(\E)$ then the corresponding $v' \in M(\E \otimes \LL)$ satisfies $D_{v'}= D_{v} \otimes \LL$.
\end{lemma}
\begin{proof}
We will think of $\LL$ as a divisor, thus we can write $\LL = \sum a_i D_i$. Then we have  that $E^\rho(j) \simeq (E \otimes L)^\rho(j+a_\rho)$. The matroid is constructed by choosing bases (in a particular manner) of all possible subspaces of the form $\cap_\rho E^\rho(j_\rho)$. By the above equality we see that under the isomorphism $E \otimes L \simeq E$ we get exactly the same subspaces for $\E \otimes \LL$ thus the matroids are equal.

Given $v \in M(\E)$ we can write $D_v = \sum b_\rho D_\rho$. Then by the above we see that $D_{v'} = \sum (a_\rho+b_\rho) D_\rho$, which equals $D_v \otimes \LL$.
\end{proof}
The above has the following consequences: Constructing the parliament of polytopes for a $\E$ is equivalent to constructing the surjection
\[ \bigoplus_{e \in M(\E)} \OO(D_e) \to \E \to 0. \]
We see by the lemma that constructing the parliament for $\E \otimes \LL$ corresponds to the tensored sequence
\[ \bigoplus_{e \in M(\E)} \OO(D_e) \otimes \LL \to \E \otimes \LL \to 0. \]
In other words the sequence will remain surjective on global sections after tensoring with any line bundle $\LL$. 

\begin{proposition} \label{prop:commDiagram}
Given a vector bundle $\E$ and line bundles $\LL_1,\LL_2$ on $X_\Sigma$ and vectors $e_1,e_2$ in the ground set of $M(\E)$, there exists a commutative diagram
\begin{equation*}
\begin{tikzcd}
 H^0(\OO(D_{e_1}) \otimes \LL_1)_u \otimes   H^0(\OO(D_{e_2}) \otimes \LL_2))_v  \arrow[r, "\phi"] \arrow[d,"f"]
&  H^0( \OO(D_{e_1e_2}) \otimes \LL_1 \otimes \LL_2))_{u+v}  \arrow[d,"g"] \\
H^0(S^k \E \otimes \LL_1)_u \otimes H^0(S^l \E \otimes \LL_2)_v \arrow[r , "\eta" ]
& H^0( S^{k+l} \E \otimes \LL_1 \otimes \LL_2)_{u+v}
\end{tikzcd}
\end{equation*}
 where $\eta$ and $\phi$ are  multiplication maps and f is the map coming from the parliament of polytopes. Also, $g$ is induced from the map of vector spaces sending  $ 1 \in \C$ to $e_1e_2$ in $S^{k+l} \E$.
\end{proposition}
\begin{proof}
We have that $\phi(\chi^u \otimes \chi^v)=\chi^{u+v}$. Moreover $g(\chi^{u+v}) = e_1e_2 \in H^0(S^{k+l} \E \otimes \LL_1 \otimes \LL_2)_{u+v} = \cap_\rho (E \otimes L_1 \otimes L_2)^\rho(\langle u,\rho \rangle)$.

Going in the other direction $f(\chi^u \otimes \chi^v)=e_1 \otimes e_2$, where we now consider $e_1 \in H^0(S^k E \otimes L_1)_u = \cap (E \otimes L_1)^\rho(\langle u,\rho \rangle) $ and  $e_2 \in H^0(S^k E \otimes L_2)_v = \cap (E \otimes L_2)^\rho(\langle u,\rho \rangle)$. We see that $\eta(e_1 \otimes e_2) = e_1e_2$.
\end{proof}

In particular all multiplication maps of global sections of line bundles on $\p(\E)$ can be lifted to multiplication of global sections of line bundles on $X_\Sigma$.

We now come to a key definition of this section, namely that of a set of vectors $\mathfrak{M}(\E)$, whose elements correspond to matroid vectors of some $S^k \E$ which are not symmetric products of matroid vectors for smaller $k$. We will subsequently see that the elements of $\mathfrak{M}(\E)$ (together with generators for $\Cox(X_\Sigma))$ correspond to generators of  $\Cox(\p(\E))$.

\begin{definition}
Given a toric vector bundle $\E$ we construct the set $\mathfrak{M}(\E)$ by induction on $k$ as follows: The step $k=1$ corresponds to adding all vectors in the ground set of $M(\E)$ to $\mathfrak{M}(\E)$. Assume now we have completed step $k-1$. Step $k$ amounts to adding all vectors of $M(S^k \E)$ which cannot be written as a symmetric product of vectors already in $\mathfrak{M}(\E)$.

For a vector $e \in \mathfrak{M}(\E)$ we denote by $\deg e$ the integer such that $e \in M(S^{\deg e} \E)$.
\end{definition}

\begin{proposition} \label{prop:surjRingmap}
Given a toric vector bundle $\E$, there exists a surjective map of $\C$-algebras
\[ \bigoplus_{t_1,\ldots,t_n \in \Z} \bigoplus_{s_i \in \Z_{\geq 0} | e_i \in \mathfrak{M}(\E)  } H^0(X_\Sigma, \OO( \sum_{i} s_i D_{e_i}+t_1 D_1+\ldots+t_n D_n)) \]
 \[ \to \bigoplus_{\substack{k \in \Z_{\geq 0}, \\ t_1,\ldots,t_n \in \Z}} H^0(X_\Sigma,S^k \E \otimes \OO(t_1D_1 + \ldots + t_n D_n)), \]
 where in each summand in the above  sum only finitely many $s_i$ are non-zero.
\end{proposition}
\begin{proof}
The first double sum is simply the ring of all sections  of all integral linear combinations  of the $T$-invariant divisors on $X_\Sigma$, as well as of all positive sums of divisors $D_{e}$, where $e \in \mathfrak{M}(\E)$.  Thus it is  a $\C$-algebra under multiplication of sections of line bundles on $X_\Sigma$. 

The map is defined as follows: We fix a summand, thus we pick $t_1,\ldots,t_n$ and some non-zero $s_i$, which we denote by $s_1, \ldots s_m$. We define $\LL_1$ as $\OO(\sum_i s_i D_{e_i})$ and $\LL_2$ as $\OO(\sum_i t_i D_i)$. We set $k=\sum_i s_i \deg e_i$. Then there is, by the assumptions on $e_i$ and Klyachko's equivalence of categories, a map of vector bundles
\[ \LL_1 \otimes \LL_2 \to S^k \E \otimes \LL_2, \]
given by sending $1 \in L_1 \otimes L_2 \simeq \C$ to $e_1^{s_1} \cdots e_m^{s_m}$. This in turn induces a map $H^0(X_\Sigma, \LL_1 \otimes \LL_2) \to  H^0(X_\Sigma,S^k \E \otimes \LL_2)$ which is the map in the statement above. We note that if $e_1^{s_1} \cdots e_m^{s_m}$ is in $M(S^k \E)$ then this map is the same as the corresponding summand in the map induced by the parliaments of polytopes construction.

Fix the numbers $k,t_1,\ldots,t_n$. They correspond to the bundle $\F = S^k \E \otimes \OO(t_1 D_1 + \ldots + t_n D_n)$. By \cref{lemma:matroidtensorl} we have that $M(\F) = M(S^k \E)$. All summands in the map
\[\bigoplus_{w \in M(\F)} H^0(X_\Sigma, \OO(D_w)) \to H^0(X_\Sigma,S^k \E \otimes \OO(t_1D_1 + \ldots + t_n D_n)) \]
induced by the parliament of polytopes for $S^k \E$ will appear in the big sum in the proposition statement, by the remark at the end of the preceding paragraph and by the construction of $\mathfrak{M}(\E)$. Thus the map is surjective for any graded piece $k,t_1,...,t_n$, hence it is surjective.

By Proposition \ref{prop:commDiagram} the map is compatible with the multiplication maps, hence the map is a map of $\C$-algebras.
\end{proof}

Motivated by the above we make the following definition.

\begin{definition}
For a natural number $k$ we define $S^k M(\E)$ as all symmetric $k$-products of vectors in the ground set of the matroid $M(\E)$. By $M(S^a \E) M(S^b \E)$ we mean all symmetric products $e_1e_2$ of vectors $e_1 \in M(S^a \E)$ and $e_2 \in M(S^b \E)$.
\end{definition}

Before coming to the characterization of being a Mori Dream Space we need the following technical lemmas.

\begin{lemma} \label{lemma:filtequal}
Assume that for  a natural number $c$ and  all integers $k_\rho$  we have the equality
\[ \Span_{a,b,j_\rho,l_\rho} \{ \cap_\rho S^a E^\rho (j_\rho) \cap_\rho S^b E^\rho(l_\rho) | j_\rho+l_\rho=k_\rho, a+b = c \} = \cap_\rho S^{c} E^\rho (k_\rho). \]
Then $M(S^{c} \E) \subset \cup_{a,b | a+b=c}  M(S^a \E) M(S^b \E)$.
\end{lemma}
\begin{proof}
The matroid $M(S^{c} \E)$ is constructed by \cite[Algorithm 3.2]{DJS} from the intersections $\cap_\rho S^{c} E^\rho (k_\rho)$. We are supposed to, in a certain order depending on the dimension of the intersection, choose a basis for the intersection, or of a quotient. By our assumption we know that we can pick a basis for each such intersection consisting of vectors of the form $v_1v_2$ where $v_1 \in M(S^a \E)$, $v_2 \in M(S^b \E)$. By elementary linear algebra we can also choose such a basis for each quotient. The proposition follows.
\end{proof}

\begin{lemma} \label{lemma:sectionnotmultiple}
Assume that for  a natural number $c$ and  fixed integers $k_\rho$  we have 
\[ \Span_{a,b,j_\rho,l_\rho} \{ \cap_\rho S^a E^\rho (j_\rho) \cap_\rho S^b E^\rho(l_\rho) | j_\rho+l_\rho=k_\rho, a+b=c \} \subsetneq \cap_\rho S^{c} E^\rho (k_\rho). \]
Then there exists a line bundle $\LL$ and $s \in H^0(S^{c} \E \otimes \LL)$ which is not in the image of the map $\bigoplus_{a,b,\LL_1,\LL_2 : \LL_1 \otimes \LL_2 = \LL} H^0(S^a \E \otimes \LL_1) \otimes H^0(S^b \E \otimes \LL_2) \to H^0(S^{c} \E \otimes \LL)$.
\end{lemma}
\begin{proof}
Let $W=\cap_\rho S^c E^\rho (k_\rho)$. We define $\LL=-D_W = \sum_\rho -k_\rho D_\rho$. Then
\[ H^0(S^c E \otimes \LL)_0 = \cap (S^c E \otimes L)^\rho(0) = \cap S^c \E^\rho(k_\rho) = W, \]
thus we obtain $W$ as the space of $T$-invariant global sections of $\E \otimes \LL$.
The image of 
\[ \bigoplus_{\substack{a,b,\LL_1,\LL_2 \\ \LL_1 \otimes \LL_2 = \LL}} H^0(S^a \E \otimes \LL_1) \otimes H^0(S^b \E \otimes \LL_2), \]
of degree $0$ can be described as follows:  First of all to map to the character $0$ we can pick any character $v$ in the first factor and $-v$ in the second factor, thus we get this sum equals 
\[ \bigoplus_{\substack{v,a,b,\LL_1,\LL_2 \\ \LL_1 \otimes \LL_2 = \LL}} H^0(S^a \E \otimes \LL_1)_v \otimes H^0(S^b \E \otimes \LL_2)_{-v}. \]
This is equal to the direct sum
\[ \bigoplus_{\substack{a,b,j_\rho',l_\rho' \\ j_\rho'+l_\rho'=k_\rho}}  \cap_\rho S^a E^\rho (j_\rho'+ \langle v,\rho \rangle ) \cap_\rho S^b E^\rho(l_\rho'- \langle v,\rho \rangle )  ,  \]
because varying  $j_\rho'$ and $l_\rho'$ corresponds exactly to varying the line bundles $\LL_1,\LL_2$. We set $j_\rho = j_\rho'+ \langle v,\rho \rangle$ and $l_\rho=l_\rho'- \langle v,\rho \rangle$. By the assumption  the image of this direct sum is contained in a space of smaller dimension than $W$, thus there must exist some $s \in W$, $s \notin \{ \cap_\rho S^a E^\rho (j_\rho) \cap_\rho S^b E^\rho(l_\rho) | j_\rho+l_\rho=k_\rho \}$. Hence we are done.
\end{proof}

\begin{remark}
The above proof shows that a linear space $V \subset E$ is in $L(\E)$ if and only if there exists a line bundle $\LL$ and a character $u$ such that $H^0(\E \otimes \LL)_u = V$.
\end{remark}

\begin{theorem} \label{theorem:criterion}
$\p(\E)$ is a Mori Dream Space if and only if the set $\mathfrak{M}(\E)$ is finite. Equivalently if and only if there exists an integer $c$ such that for any $k \geq c$ we have that
\[M(S^k \E) \subset \cup_{a+b=k} M(S^a \E) M(S^b \E). \]
\end{theorem}
\begin{proof}

By the way we constructed $\mathfrak{M}(\E)$ it is clear that if such a $c$ exists then $\mathfrak{M}(\E)$ is finite.

Assume that $\mathfrak{M}(\E)$ is finite. Then by \cref{prop:surjRingmap} we have that $\Cox(\p(\E))$ is the image of the section ring of finitely many line bundles on $X_\Sigma$, hence it is finitely generated.

Lastly we will prove that if no such $c$ exists, then $\Cox(\p(\E))$ cannot be finitely generated. Since no such $c$ exist we must have an infinite sequence $c_1,c_2,c_3,\ldots$ satisfying
\[ M(S^{c_i} \E) \not \subset \cup_{a+b=c_i} M(S^a \E) M(S^b \E) \]
For any fixed such $c_i$ we see by  \cref{lemma:filtequal} that the assumptions of \cref{lemma:sectionnotmultiple} has to be satisfied for some $k_\rho^i$. Set $\LL_i = \sum_\rho -k_\rho^i D_\rho$. By \cref{lemma:sectionnotmultiple} there exists  a  section $s_i \in  H^0(S^{c_i} \E \otimes \LL_i)$ which cannot be in the algebra generated by sections of bundles of the form $S^k \E \otimes \LL$, for $k< c_i$. Thus to generate the Cox ring there has to be at least one generator which is a section of some $S^{c_i} \E \otimes \LL $. Since this is true for any $c_i$, of which there are infinitely many, there cannot be a finite generating set for $\Cox(\p(\E))$. This concludes the proof.
\end{proof}

\section{Examples of Cox rings} \label{posSection:exCoxRings}

In this sections we apply the results in the preceding section to certain classes of vector bundles. We begin with an easy observation:

\begin{proposition} \label{symmetricPowerMDS}
Assume that $M(S^k \E) \subset S^k M(\E)$ for all $k$. Then $\p(\E)$ is a Mori Dream Space.
\end{proposition}
\begin{proof}
This follows from Theorem \ref{theorem:criterion}.
\end{proof}

\begin{proposition} \label{prop:symPower1dim}
Assume that $\E$ is a toric vector bundle of rank $r$ such that all subspaces appearing in the filtrations are either of dimension $0,1,r$. Then $M(S^k \E) \subset S^k(M(\E))$, and in particular $\p(\E)$ is a Mori Dream Space.
\end{proposition}
\begin{proof}
It is clearly sufficient to consider only the rays which have a one-dimensional space in the filtration. Let $v_1,\ldots,v_s$ be the vectors in these one-dimensional spaces, thus $M(\E)= \{v_1,\ldots,v_s \}$.

Let the filtration on ray $\rho_i$ be given by 
\[ E^i(j) = \begin{cases} 
E &\text{ if } j \leq a_i \\
v_i &\text{ if } a_i < j \leq b_i \\
0 &\text{ if } \hspace{5 pt} b_i < j
\end{cases}. \]
Then the filtration of $S^k \E$ is given by
\[ S^k E^i(j) = \begin{cases} 
S^k E &\text{ if } j \leq ka_i \\
v_iw&\text{ if } ka_i < j \leq (k-1)a_i+b_i \\
\vdots \\
v_i^{k-1}w &\text{ if } 2a_i+(k-2)b_i < j \leq a_i+(k-1)b_i \\
v_i^k &\text{ if } a_i+(k-1)b_i < j \leq kb_i \\
0 &\text{ if } kb_i < j
\end{cases}, \]
where in each step $w$ can be any vector not in the span of $v_i$. Thus to compute the matroid we need to consider the spaces
\[ \cap_i S^k E^i(j_i) = \{ f=v_i^{l_i}z_i \in S^k E \text{ where }  (k-l_i+1)a_i+(l_i-1)b_i < j_i \leq (k-l_i)a_i+l_ib_i \}, \]
where, for each $i$, $z_i$ can be any vector in $S^{k-l_i} E$. If we can show that each such space is spanned by symmetric powers of $\{v_1,\ldots,v_s\}$ then we are done. We see that an $f$ in this intersection corresponds to a hypersurface $f$ of degree $k$ in $\p(E)$ vanishing to order $l_i$ at the hyperplane $H_i=V(v_i)$. But if $f$ vanishes to order $l_i$ at $H_i$ then $f=v_i^{l_i}g$ for some hypersurface $g$ of degree $k-l_i$. Iterating we obtain that $f=v_1^{l_1} v_2^{l_2} \cdots v_s^{l_s}$, which implies exactly what we wish to prove.

\end{proof}

The above give new proofs of Theorem 5.7 and Theorem 5.9 in \cite{HS}:

\begin{corollary}
If $\E$ has rank $2$ or if $\E$ is a tangent bundle, then $\p(\E)$ is a Mori Dream Space.
\end{corollary}

\section{A presentation of the Cox ring} \label{posSection:coxPresentation}

We here give a presentation of the Cox ring of $\p(\E)$, in terms of generators and relations, defined from the matroids $M(S^k \E)$. The essential idea is that the relations between matroid vectors in $M(S^k \E)$ correspond to all relations between sections of line bundles of the same degree $\OO_{\p(\E)}(k)$. Moreover all relations between sections which are products of sections of different $\OOO$-degree can be described in terms of the difference between $S^k M(\E)$ and $M(S^k \E)$. Using this we obtain a presentation of the Cox ring of any projectivized toric vector bundle.

To each ray $\rho \in \Sigma(1)$ we associate a variable $S_\rho$. These correspond to pullbacks of all generators for $\Cox(X_\Sigma)$ and will be some of the generators of $\Cox(\p(\E))$.   To each vector $v \in \mathfrak{M}(\E)$ we associate a variable $T_v$. Each $T_v$ will also be a generator of $\Cox(\p(\E))$. The variable $T_v$ correspond to the following element in $\Cox(\p(\E))$: Write $D_v = \sum_\rho a_\rho D_\rho$. Consider the bundle $S^{\deg v} \E \otimes \OO(-D_v)$. Then $v$ will be in $M(S^{\deg v} \E \otimes \OO(-D_v))$ under the isomorphism $E \otimes \C \simeq E$, by \cref{lemma:matroidtensorl}. The associated divisor $D'_v$ is the trivial divisor, hence it has a unique global section. Its image in $H^0(X_\Sigma,S^{\deg v} \E \otimes \OO(-D_v))$ under the map given by the parliament of polytopes is the section which we label $T_v$.

We will show that $S_\rho$ and $T_v$ are all generators of $\Cox(\p(\E))$. Moreover we will also describe the ideal of relations between these generators. It is a sum of two ideals $I$ and $J$, which we now introduce.

Let $\Rel(S^k \E)$ denote the space of linear relations between vectors in $M(S^k \E)$. For any relation $r = \sum_i \lambda_i v_i$ in $\Rel(S^k \E)$ we associate the divisor 
\[ D_r = \sum_\rho \min_i \{ \alpha_{\rho i } | \lambda_i \neq 0 \} D_\rho \defeq \sum_\rho \beta_\rho^r D_\rho. \]
Associated to relations between matroid vectors for a fixed symmetric power $k$ we get relations in the Cox ring given by the ideal
\[ I= \big\langle \sum_i \lambda_i S^i_r T_{v_i} | r=\sum \lambda_i v_i \in \Rel(S^k \E) \big \rangle, \]
where $S_r^i = \prod_\rho S_\rho^{\alpha_{\rho i} - \beta_\rho^r}$.

The reason the polynomials in the ideal $I$ give relations in the Cox ring is the following: For simplicity we consider what happens  when $k=1$, in other words for $M(\E)$. Any $v \in M(\E)$ defines a divisor $D_v$ and thus a polytope $P_v$ giving global sections. This polytope could be empty or full-dimensional or anything in between. However if we consider $\F = \E \otimes \OO(-D_v)$ then by construction $v \in M(\F)$ and the associated divisor $D_v'$ is the trivial divisor. In particular it has a unique global section. 

If we now fix a non-trivial relation $r = \sum_i \lambda_i v_i$ in $\Rel(\E)$, we have for any $v_i$ a unique associated global section. These are the sections corresponding to the variables $T_{v_i}$. Multiplying $T_{v_i}$ with a monomial in the $S_\rho$ corresponds to tensoring with effective divisors which are pull-backs from the base. In particular, the polytope will become bigger. The relations in the ideal $I$ are relations obtained from multiplying each variable with  monomials in $S_\rho$, such that the polytopes corresponding to $v_i$ overlap. As soon as they overlap, there must be a relation between the corresponding sections by \cref{proposition:dimglobalchru}.

\begin{example} \label{example:coxtp2} 
Consider the projectivization of the tangent bundle $T_{\p^2}$ of $\p^2$. We denote by $D_0,D_1,D_2$ the $T$-invariant divisors.  The parliament of polytopes is shown in \cref{figure:parliamentTP22}. The matroid has three divisors, equal to  $D_0,D_1$ and $D_2$. Thus there are three $T$-variables in $\Cox(\p(T_{\p^2}))$: $T_0,T_1,T_2$. $T_i$ is the unique section of $\OO_{T_{\p^2}}(1) -D_i$ of weight $0$.

Since the three matroid vectors are linearly dependent, we know that there will be a relation in the Cox ring involving $T_0,T_1,T_2$. The relation will be given by multiplying up with $S_i$ variables (corresponding to the $D_i$ as pullbacks of divisors from the base) until the three polytopes overlap and are of the same degree in $\Pic(\p(T_{\p^2}))$. This is in this case exactly of degree $\OO_{T_{\p^2}}(1)$ and is shown in the image. Thus the associated relation is
\[ T_0S_0 + T_1S_1+T_2S_2.\]
\end{example}

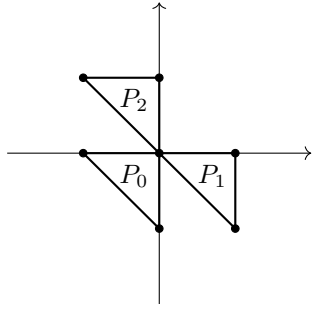
\begin{figure}
\centering
\begin{tikzpicture} [scale=1]
\node[draw,circle,inner sep=1pt,fill] at (0,0) {};
\node[draw,circle,inner sep=1pt,fill] at (1,0) {};
\node[draw,circle,inner sep=1pt,fill] at (1,-1) {};
\node[draw,circle,inner sep=1pt,fill] at (0,1) {};
\node[draw,circle,inner sep=1pt,fill] at (-1,1) {};
\node[draw,circle,inner sep=1pt,fill] at (-1,0) {};
\node[draw,circle,inner sep=1pt,fill] at (0,-1) {};
\draw [->] (0,0) -- (2,0);
\draw [->] (0,0) -- (0,2);
\draw [-] (0,0) -- (-2,0);
\draw [-] (0,0) -- (0,-2);
\draw [thick,-] (0,0) -- (1,0);
\draw [thick,-] (0,0) -- (1,-1);
\draw [thick,-] (1,0) -- (1,-1);

\node at (0.70,-0.3) [] {$P_1$};
\node at (-0.33,-0.3) [] {$P_0$};
\node at (-0.33,0.7) [] {$P_2$};

\draw [thick,-] (0,0) -- (0,1);
\draw [thick,-] (0,0) -- (-1,1);
\draw [thick,-] (0,1) -- (-1,1);

\draw [thick,-] (0,0) -- (-1,0);
\draw [thick,-] (0,0) -- (0,-1);
\draw [thick,-] (0,-1) -- (-1,0);
\end{tikzpicture}
\caption{The 3 polytopes of the parliament of $T_{\p^2}$.}
 \label{figure:parliamentTP22}
\end{figure}

\begin{remark}
Any $v \in M(S^k \E)$ is either in $\mathfrak{M}(\E)$ or is a symmetric product of elements in $\mathfrak{M}(\E)$. If it is the latter then by the letter $T_v$ we denote the monomial $\prod_{i=1}^l T_{v_i}$ where $v$ equals the symmetric product $v_1 \cdots v_l$, with $v_i \in \mathfrak{M}(\E)$. Thus, for any $v \in M(S^k \E)$ the symbol $T_v$ correspond to a distinguished element in $\Cox(\p(\E))$.
\end{remark}

In addition to the relations in $I$, we also get relations between generators of different degrees. Fix $v_1 \in M(S^a \E)$ and $v_2 \in M(S^b \E)$ and consider the associated sections: They are $T_{v_1} \in H^0(S^a \E \otimes \OO(-D_{v_1}))_0$ and $T_{v_2} \in H^0(S^b \E \otimes \OO(-D_{v_2}))_0$, respectively. Then the product $T_{v_1}T_{v_2}$ will lie in $H^0(S^{a+b} \E \otimes \OO(-D_{v_1} -D_{v_2}))_0$. Choosing a basis  $\{ w_1,\ldots,w_s\} \subset M(S^{a+b} \E \otimes \OO(-D_{v_1} -D_{v_2}))$  of the latter space, we can write $v_1v_2 = \sum a_i w_i$. Then we get the relation of sections 
\[ T_{v_1} T_{v_2} = \sum_i a_i T_{w_i} S_{w_i}, \]
where $T_{w_i} \in H^0(S^{a+b} \otimes \OO(-D_{w_i}))_0$ and $S_{w_i}$ is defined as follows: We can write $D_{w_i}-D_{v_1}-D_{v_2} = \sum m_\rho D_\rho$. By assumption this is an effective divisor (since $v_1v_2$ lies in the intersection spanned by $w_i$) thus we can define $S_{w_i} = \prod_\rho S_\rho^{m_\rho}$. We let $J$ be the ideal generated by all relations such as this.  The ideal $J$ in some sense measures the difference between $M(S^k \E)$ and $S^k M(\E)$; if $(S^k \E)=S^k M(\E)$ for all $k$, then the ideal $J$ is empty.

Even though we defined the generators of $J$ in terms of symmetric products of two matroid vectors, the following lemma shows that they imply that the analogous relations for symmetric products of arbitrarily many vectors lies in $J$.

\begin{lemma}
Given $v_1,\ldots,v_p$ with $v_i \in M(S^{a_i} \E)$ and let $\{w_1,\ldots,w_s\}$ be a basis of $H^0(S^{\sum a_i} \E \otimes \OO(\sum_i -D_{v_i}))_0$ and write $v_1 \cdots v_p = \sum b_i w_i$. Then the relation
\[ \prod_i T_{v_i}-\sum_i b_i T_{w_i}S_{w_i} \]
lies in $J$, where $S_{w_i}$ is some monomial in $S_\rho$, defined as above.
\end{lemma}
\begin{proof}
We will prove this for $p=3$; the general case follows by iterating the process. Thus we assume that $v_1v_2v_3 = \sum b_iw_i$, where $w_i \in M(S^3 \E)$.

Let $v_1v_2 = \sum c_k z_k$, where $z_k \in M(S^2 \E)$, thus we have an associated relation in $J$
\[ T_1T_2 = \sum_k c_k T_{z_k}S_{z_k}. \]
Multiplying the above by $T_3$ we get
\[ T_1T_2T_3 = \sum_k c_k T_{z_k}T_3S_{z_k}. \]
For a fixed $k$ write $z_kv_3= \sum d_{k,j} w_j$, thus we have an associated relation
\[ T_{z_k}T_3 = \sum_j d_{k,j} T_{w_j} S_{k,j}. \]
Substituting this into the above we get
\[ T_1T_2T_3 = \sum_k c_k S_{z_k} \sum_j d_{k,j} T_{w_j} S_{k,j} = \sum_j T_{w_j}(\sum_k c_k d_{k,j} S_{z_k} S_{k,j}). \]
We have that by construction $S{z_k}$ corresponds to the effective divisor $D_{z_k}-D_{v_1}-D_{v_2}$ and $S_{k,i}$ corresponds to $D_{w_i}-D_{z_k}-D_{v_3}$, hence their product $S_{z_k}S_{k,i}$ corresponds to the sum $D_{w_i}-D{v_1}-D_{v_2}-D_{v_3} = S_{w_i}$. Observe also that by combining the expressions for $v_1v_2v_3$ we see that $b_i = \sum_k c_k d_{k,i}$, thus the relation above simplifies to
\[ T_1T_2T_3 = \sum_i b_i T_{w_i}S_{w_i}, \]
which is what we wanted to show.
\end{proof}

\begin{theorem} \label{theorem:coxpresenetation}
The Cox ring of $\p(\E)$ is
\[ \C[S_\rho,T_v |  / (I+J). \]
\end{theorem}
\begin{proof}
From the fact that the Cox ring of the toric variety is $\C[S_\rho]$ \cite{CoxRingOriginal} and the existence of the surjections
\[\bigoplus_{v_i \in M(S^k \E} \OO(D_{v_i}) \otimes \LL \to S^k \E \otimes \LL,\]
which are surjective on global sections, it is clear that $S_\rho$ and $T_v$ generate the Cox ring of $\p(\E)$, thus we need to determine the relations. Since $\Cox(\p(\E))$ is graded by $\Pic (\p(\E))$ we must have that any relation is between elements of the same class in the Picard group, i.e. between elements of some fixed class $\OO_{\p(\E)}(k) + \sum_\rho t_\rho D_\rho$. We also have a finer grading given by the character $u$ of the torus. After tensoring with $\ddiv(-u)$ we may assume that $u=0$. First we show that both $I$ and $J$ is contained in the ideal of the Cox ring.

Fixing a relation $r = \sum_i \lambda_i v_i$ in $\Rel(S^k \E)$ we have, for each $i$, an associated trivial divisor $D'_{v_i}$ of $S^k \E \otimes \OO( -D_{v_i})$ and a section in $H^0(S^k \E \otimes \OO (-D_{v_i}))_0$ which corresponds to $T_{v_i}$.  By construction $S_r^i$ corresponds to the effective divisor $\sum_\rho (a_{\rho i} - \beta_\rho^r) D_\rho$, thus a section in $H^0(\E)_0$. Multiplying $T_{v_i}$ and $S_r^i$ together gives a section of $S^k \E \otimes \OO ( \sum_\rho -\alpha_{\rho i} D_\rho + \sum_\rho (\alpha_{\rho i}-\beta_\rho^r) D_\rho) = S^k \E \otimes \OO ( -D_r)$ of weight $0$. 

We have that $\lambda_i T_{v_i} S_r^i$ corresponds to the section $ \lambda_i v_i \in  (S^k E  \otimes \OO (-D_r))^\rho(0) = H^0(S^k \E( -D_r))_0$. Thus $\sum_i \lambda_i T_{v_i} S_r^i$ corresponds to $ \sum_i \lambda_i v_i \in H^0(S^k \E \otimes \OO (-D_r))_0$, which is zero, thus $\sum_i \lambda_i T_{v_i} S_r^i$ must be a relation, so $I$ is contained in the Cox ideal.

Fixing a generator $T_{v_1} T_{v_2} - \sum_i a_i T_{w_i} S_{w_i}$ of $J$ (as defined above) we see that under the isomorphism (since $w_i$ is a basis for the latter space)
\[ \oplus_i H^0(\OO (D_{w_i}))_0 \to H^0(S^{a+b} \E \otimes \OO (-D_{v_1}  -D_{v_2}))_0 \]
$v_1v_2$ is sent to the same as $\sum_i a_i w_i$ thus the generator has to be a relation in the Cox ring. Thus $J$ is contained in the Cox ideal.

Conversely, assume that we are given a relation between sections in $\Cox(\p(\E))$ of multidegree $(k,l_1,\ldots,l_s)$ and weight $0$, in other words in $V=H^0(S^k \E  \otimes \OO (l_1D_1 + \cdots + l_sD_s))_0$. We can write the relation as
\[ \sum_{i \in K} c_i \prod_j T_{v_{i,j}}^{n_{i,j}} \prod_\rho S_\rho^{m_{\rho,i}} = 0, \]
where $K$ is the index-set of all monomials appearing in the relation. Choose a basis $w_1,\ldots,w_q$ of $V$, where $w_i \in M(S^k \E)$.

Fix $i$ and write $z_j = \prod_j v_{i,j}^{n_{i,j}} = \sum_{i=1}^q a_{i,j} w_i$. Then the relations in $I$ and $J$ imply that we have the relation
\[ \prod_j T_{v_{i,j}}^{n_{i,j}} = \sum_{i=1}^q a_{i,j} T_{w_i} S_{i,j} \]
where $S_{i,j}$ is some monomial in $S_\rho$. Thus we can replace $\prod_j T_{v_{i,j}}^{n_{i,j}}$ in the relation with the sum above.  Doing this for all monomials $z_j$ we get a polynomial whose only $T$-variables appearing are $T_{w_i}$, in other words a polynomial $\sum_{i=1}^q b_i T_{w_i} S_i$, where $S_i$ is some monomial in $S_\rho$. This is by assumption equal to $0$, however it corresponds to the element $\sum_{i=1}^q b_iw_i$ in $V$. Since the $w_i$ by construction form a basis for $V$, this forces $b_i=0$ for all $i$. Thus the relation we started with is a sum of relations coming from $I$ and $J$.

\end{proof}

\begin{remark}
Note that in the above we do not assume that $\Cox(\p(\E))$ is finitely generated.
\end{remark}

\begin{remark}
This generalizes the presentations of the Cox rings of rank two bundles and tangent bundles given in \cite[Theorem 5.7, Theorem 5.9]{HS}.
\end{remark}

\begin{example}
Consider again the tangent bundle of $\p^2$. By \cref{prop:symPower1dim} the ideal $J$ is empty. Thus by \cref{example:coxtp2} 
\[\Cox(\p(T_{\p^2})) = \C[T_0,T_1,T_2,S_0,S_1,S_2]/(T_0S_0 + T_1S_1+T_2S_2) . \]
This is consistent with the well-known description of $\p(T_{\p^2})$ as a $(1,1)$ divisor in $\p^2 \times \p^2$, where the defining equation is a consequence of the Euler sequence.
\end{example}

\begin{example}
From the above we can recover all results on Cox rings of toric vector bundles found in \cite{GHPS}. For example let $\E$ be a bundle with filtrations given as
\[ E^\rho(j) = \begin{cases} 
E &\text{ if } j \leq 0 \\
V_\rho &\text{ if } j=1 \\
0 &\text{ if } 1 < j
\end{cases} \]
where $\dim V_\rho>1$ which is what is mostly studied in \cite{GHPS}. Let $Y$ be the iterated blowup of $\p(E)$ in all linear spaces in $L(\E)$; first blow up points (corresponding to hyperplanes $V_\rho$), then strict transform of lines and so on. Let $X$ be $Y$ minus all exceptional divisors over linear subspaces not equal to some $V_\rho$. By \cite[Theorem 3.3]{GHPS} the Cox ring of $\p(\E)$ is isomorphic to the Cox ring of $X$. We can see this as follows from our theorem: Effective divisors of some $S^k \E \otimes \OO(D)$ correspond exactly to $f \in \cap (S^k E \otimes D)^\rho(0)$, for some $D=\sum_\rho a_\rho D_\rho$. In other words to degree $k$ polynomials $f$ on $\p(E)$ vanishing to order $a_\rho$ at $V_\rho$. Such effective divisors which are not symmetric products of divisors of lower $\OOO$-degree, correspond exactly to polynomials $f$ which cannot be written as a product of lower degree such polynomials. This corresponds exactly to generators of the Cox ring of the associated blow-up of $\p(E)$. By well-known results on generators of Cox rings of such blow-ups, these are often not finitely generated, for instance if $V_\rho$ consist of $9$ or more general hyperplanes \cite{Mukai}.

For example we can pick $V_\rho$ to be 5 general hyperplanes in $E \simeq \C^3$, corresponding to five general points of $\p^2 = \p(E)$. Then the Cox ring is generated by sections pulled back from the base, together with the ten lines between the corresponding points of $\p(E)$, corresponding to the fact that $M(\E)$ has ten elements, and by the unique quadric passing through the five points; this corresponds to the fact that there is one unique intersection of the filtrations for $S^2 \E$ which is not the span of symmetric products of vectors in $M(\E)$.
\end{example}

\begin{remark}
The above results is satisfying from a theoretical point view, since we are able to completely describe the Cox rings of $\p(\E)$ in terms of the matroids of $S^k \E$. However, in practice these results are not necessarily as satisfying, since for a specific bundle $\E$ it is not easy to say what the relationship between $M(S^k \E)$ and $S^k M(\E)$ is for any $k$:  It is at least as hard as describing generators of Cox rings of general iterated blow-ups of projective space in linear spaces, which is well-known to be a hard problem.

The above results does significantly improve our understanding of sections of line bundles on $\p(\E)$. Elements of $\mathfrak{M}(\E)$ correspond exactly to sections of some such line bundles, which is not in the algebra generated by sections of lower $\OOO$-degree. We will in the next section that this has the consequence that to check whether $\E$ is big we need also to consider the matroids of all $S^k \E$ at once. 
\end{remark}

\section{A bigness criterion} \label{posSection:big}

Let $\E$ be a toric vector bundle on the toric variety $X_\Sigma$. We say that $\E$ is big if $\OO_{\p(\E)}(1)$ is a big line bundle. This is what Jabbusch calls $L$-big \cite{Jabbusch}. In this section we investigate when a toric vector bundle is big.

\begin{example}
If $\E$ is a direct sum of line bundles then $\E$ is big if and only if some positive linear combination of the line bundles is big \cite[Lemma 2.3.2]{Laxarsfeld1}. Thus $\OO(-1) \oplus \OO(2)$ on $\p^d$ is big, but not nef.
\end{example}

Di Rocco, Jabbusch and Smith ask whether a toric vector bundle is big if and only if some Minkowski sum of the polytopes in the parliament is full-dimensional \cite[p.3]{DJS}. The following example shows that this is not the case.

\begin{example}  \label{example:bignominkowski}
Let $X=\p^1 \times \p^1$ and denote by $H_1$ and $H_2$ the pullbacks of the hyperplane classes of each factor. Let $\E = \OO_X(H_1) \oplus \OO_X(H_2-H_1)$. Then since $D=H_1+H_2 = 2H_1 + (H_2-H_1)$ is big, we see by the previous example that $\E$ is big. However the parliament of polytopes is the following: $P_{H_1}$ is a line segment of length $1$ while $P_{H_2-H_1}$ is empty. Thus no Minkowski sum of the polytopes in the parliament is full-dimensional.
\end{example}
 For divisors $D$ and $E$ on a toric variety we have the inclusion $P_D + P_E \subset P_{D+E}$, however in general this is not an equality. This is the reason for the above example giving a big vector bundle even if no Minkowski sum is full-dimensional. To rectify this we might ask the similar question which still might be true:
\begin{question}
Is $\E$ big if and only if a positive linear combination of the divisors in the parliament is big?
\end{question}
The following example shows that also this is not true in general:

\begin{example} \label{example:big}
Let be $\p^2$ be given as a toric variety as the complete fan with ray generators $\rho_0 = -e_1-e_2, \rho_1=e_1, \rho_2= e_2$. Consider the rank 3 bundle on $\p^2$ given by Klyachko filtrations for $i=1,2$:
\[ E^i(j) = \begin{cases} 
E &\text{ if } j \leq -1 \\
W_i &\text{ if } j=0 \\
v_i &\text{ if } j = 1 \\
0 &\text{ if } 1 < j
\end{cases}, \]
and for $i=0$
\[ E^0(j) = \begin{cases} 
E &\text{ if } j \leq 0 \\
W_0 &\text{ if }  j = 1 \\
v_0 &\text{ if }  j = 2 \\
0 &\text{ if } 1 < j
\end{cases}. \]
Here $v_i$ are general vectors of $E \simeq \C^3$ and $W_i$ is a general $2$-dimensional vector space containing $v_i$. Let $l_{ij} = W_i \cap W_j$. Then $M(\E)= \{ v_i, l_{ij} \}$ and each associated divisor is trivial. Hence no positive sum of these can be big. However, considering $S^2 \E$ we have that the intersection $S^2E^0(1) \cap S^2E^1(0) \cap S^2E^2(0)$ is one dimensional, so there is a matroid vector $w$ in this intersection. We see that $D_w = \OO_{\p^2}(1)$, thus it is big, which implies that $H^0(X_\Sigma,S^{2k} \E) = H^0(\p(\E),\OO_{\p(\E)}(2k))$ grows as $k^{\dim \p(\E)}$. This means that $\E$ is big.
\end{example}
In the example above we have that $\E$ is big even if no positive sum of divisors in the parliament is big. However we have that there is a big divisor in the parliament of $S^k \E$ for all sufficiently large  $k$ (in the example all $k \geq 2$). The following theorem shows that this will always happen.

\begin{theorem} \label{theorem:bignessCriteria}
A toric vector bundle is big if and only if there exists $k>0$ and a vector $v \in M(S^k \E)$ such that $P_v$ is full dimensional.
\end{theorem}

\begin{proof}
By \cite[Example 6.1.23]{Lazarsfeld2} we have that $\E$ is big if and only if for some (every) ample $A$ we have that $H^0(S^k \E \otimes A^{-1}) \neq 0$ for some $k>0$.

\iffalse Assume $k,v$ exists and fix some ample $A= \sum_\rho a_\rho D_\rho$ and let $n=\max a_\rho$.
Since $P_{D_v}$ is full dimensional we must be able to choose, for sufficiently large $l$ an inner lattice point $p \in l P_{D_v}$ with distance at least $n$ to each hyperplane $\rho^\perp$ in other words such that we have
\[ \langle p,\rho \rangle \leq b_\rho -n \]
where $lD_v = \sum_\rho b_\rho D_\rho$. Let now $D=lD_v -A$. Then we see that for every ray $\rho$ we have
\[ \langle p,\rho \rangle \leq b_\rho -n \leq b_\rho -a_\rho \]
thus $p \in P_D$. Now there must be a matroid vector $w \in M(S^{kl} \E \otimes A^{-1})$ such that $P_D \subset P_w$, thus $p$ gives a nontrivial global section of $D_w$ which maps to a nontrivial global section of $S^{kl} \E \otimes A^{-1}$. \fi

Assume $\E$ is not big and let $A$ be an ample line bundle. Then for every $k>0$ we have that $H^0(S^k \E \otimes A^{-1}) = 0$. That means that for any $v \in M(S^k \E)$ we have $H^0(\OO(D_v) \otimes A^{-1}) = 0$, since the map $\OO(D_v) \otimes A^{-1} \to S^k \E \otimes A^{-1}$ is injective on global sections. There is an induced map $\OO(lD_v) \otimes A^{-1} \to S^{kl} \E \otimes A^{-1}$ which is non-zero, since it corresponds to the map of vector spaces sending $1$ to $v^l$. If there exists $l>0$ such that  $H^0(\OO(lD_v) \otimes A^{-1} ) \neq 0$ then we see that also the induced map on global sections is non-zero, thus $H^0(S^{kl} \E \otimes A^{-1}) \neq 0$, contradicting the fact that $\E$ is  not big. Thus $H^0(\OO(lD_v) \otimes A^{-1} ) =0$ for all $l$, thus each $D_v$ is not big and thus each polytope $P_v$ is not full dimensional.

Conversely assume that no such $v$ exists. That means that for each $v \in M(S^k \E)$ the polytope $P_v$ is not full dimensional. Since $P_v$ is not full dimensional $D_v$ is not big, thus $H^0(\OO(D_v) \otimes A^{-1} ) =0$. Thus $H^0(S^k \E \otimes A^{-1}) = 0$ for every $k$, thus $\E$ is not big.
\end{proof}

\section{Positivity and concave support functions} \label{section:positivity}

  A fundamental result on positivity of divisors on toric varieties is the following equivalences.

\begin{theorem} [{\cite[Theorem 6.1.7, Lemma 6.1.13, Theorem 6.1.14]{CLS}}] \label{prop:concaveSF}
Given a divisor $D = \sum_\rho a_\rho D_\rho$ on a toric variety, the following conditions are equivalent
\begin{itemize}
    \item $D$ is nef
    \item $m_\sigma \in P_D$ for all $\sigma$
    \item $\phi_D$ is concave
    \item $\phi_D(v) \geq \langle m_{\sigma'},v \rangle$ for all $\sigma'$ not containing $v$.
\end{itemize}
The following conditions are also equivalent
\begin{itemize}
    \item $D$ is ample
    \item $\phi_D$ is strictly concave
    \item $\phi_D(v) > \langle m_{\sigma'},v \rangle$ for all $\sigma'$ not containing $v$.
\end{itemize}
\end{theorem}

As for line bundles, one can ask for various positivity properties of vector bundles. These properties are often defined in terms of  corresponding properties for the line bundle $\OOO$ on the associated projective bundle $\p(\E)$ of rank one quotients of $\E$. One has that positivity notions which coincide for line bundles do not necessarily coincide for vector bundles.

Following Hartshorne \cite{HartshorneAmple} we say that a vector bundle is nef, ample, very ample respectively  if $\OOO$ is a nef, ample, very ample line bundle, respectively. On a toric variety, there are well known  criteria for checking whether a toric vector bundle is positive:

\begin{theorem} [{\cite[Theorem 2.1]{HMP}}] \label{HMPample}
$\E$ is ample (resp. nef) if and only if $\E|_C$ is ample (resp. nef) for all invariant curves $C \subset X_\Sigma$.
\end{theorem}

\begin{remark}
For checking the ampleness of a line bundle it is sufficient to restrict to invariant curves in $X_\Sigma$ which are extremal in the Mori cone of curves, since invariant curves generate the Mori cone \cite[Theorem 6.3.20]{CLS} and since the ample cone is the interior of the dual of the Mori cone. In the case of higher rank vector bundles it is easy to construct examples showing  that it is not sufficient to restrict to extremal curves in $X_\Sigma$.
\end{remark}

\begin{theorem} [{\cite[Theorem 1.2]{DJS}}] \label{theorem:gg}
A toric vector bundle $\E$ is globally generated if and only if for all maximal cones $\sigma$, where $\E|_{U_\sigma} = \oplus_{i=1}^r \OO(\ddiv(u_i))$, there exists linearly independent $e_1,\ldots,e_r \in M(\E)$ such that $u_i \in P_{e_i}$ for all $i$.
\end{theorem}

\begin{theorem} [{\cite[Corollary 6.7]{DJS}}]
A toric vector bundle $\E$ is very ample if and only if, for all maximal cones $\sigma$, where $\E|_{U_\sigma} = \oplus_{i=1}^r \OO(\ddiv(u_i))$, there exist linearly independent $e_1,\ldots,e_r \in M(\E)$ such that $u_i \in P_{e_i}$ for all $i$ and such that the following is true: Each $u_i$ is automatically a vertex of $P_{e_i}$. We require that the edges of $P_{e_i}$ emanating from the vertex $u_i$ generate a cone which is a translate of $\sigma^\vee$ for all $i$.
\end{theorem}

We will now reinterpret the criterion to be nef/ample in terms of concave support functions, which will lead us to a result of similar spirit to \cref{prop:concaveSF} for higher rank  toric vector bundles.

\subsection{Branched covers of fan}

We here briefly recall Payne's notion of the branched cover of a fan which is associated to $\E$ \cite{PayneBranched}. Payne gives a systematic treatment of cone complexes and how to associate a cone complex which is a branched cover of the fan $\Sigma$ to a toric vector bundle $\E$. We will a bit more naively construct the branched cover as in \cite[Example 1.2]{PayneBranched}.

We are given a toric variety $X_\Sigma$ and a toric vector bundle $\E$ on $X_\Sigma$ and we will construct a topological space $\boldsymbol{\Sigma}_\E$ together with projection map $f: \boldsymbol{\Sigma}_\E \to \Sigma$ with the structure of a rank $r$ branched cover. This means that for any cone $\sigma$ the inverse image of $\sigma$ under $f$ is  isomorphic to $r$ copies of $\sigma$, counted with multiplicity. We will also construct an associated piecewise linear support function $\Psi_\E$ on $\boldsymbol{\Sigma}_\E$.

For each maximal cone $\sigma_i$ we can consider the restriction $\E|_{U_{\sigma_i}} = \sum_{k=1}^r \OO(\ddiv u_{ik})$ where $u_{ik}$ are characters of the torus. For each of the $r$ summands we take a copy of $\sigma_i$, denoted by $\sigma_{ik}$. We define $\boldsymbol{\sigma_i}_\E$ to be $\coprod_k \sigma_{ik}/\sim$ where the equivalence relation $\sim$ is given as follows: If $\tau$ is a face of $\sigma_i$ such that $u_{ik} = u_{il}$ on $\tau$ we identify the corresponding copies of $\tau$: $\tau_{ik} \sim \tau_{il}$. The piecewise linear function $\Psi_\E$ is linear on each cone $\sigma_{ik}$, given by $v \mapsto \langle u_{ik},v \rangle$.

The topological space $\bE$ is given by $\coprod_i \boldsymbol{\sigma_i}_\E / \equiv$ where the equivalence relation $\equiv$ is given as follows. Fix a cone $\tau$ of codimension one, it is contained in two maximal cones $\sigma_1$ and $\sigma_2$. Set $\E|_{U_{\sigma_1}} = \sum_k \OO(\ddiv(u_{1k}))$ and $\E|_{U_{\sigma_2}} = \sum_k \OO(\ddiv(u_{2k}))$. Then $\tau$ corresponds to a curve $C$  which is isomorphic to $\p^1$ and any vector bundle on $\p^1$ splits as a sum of line bundles. A result by Hering, Mustata and Payne shows that for an invariant curve $C$, $\E|_C$ splits equivariantly as a sum of line bundles. Moreover, the numbers $a_j$ such that the splitting type on $C$ is $\oplus_j \OO_{\p^1}(a_j)$, correspond to unique pairs of characters $(u_{1k},u_{2l})$:  $u_{1k}-u_{2l} = a_j \rho_\tau$, where by $\rho_\tau$ we mean the primitive generator of $\tau^\perp$ positive on $\sigma$. Thus, up to permutation, we may assume that $u_{1k}$ is paired with $u_{2k}$ for all $k$. We then glue $\sigma_{1k}$ to $\sigma_{2k}$ via identifying the face $\tau_{1k}$ with the face $\tau_{2k}$.

In other words if one is in a fixed sheet $\sigma_{1k}$ above a maximal cone $\sigma$ and moves towards a neighbouring maximal cone $\sigma_2$ through a cone of codimension one $\tau$, then the sheet $\sigma_{2l}$ you end up in is the one corresponding to the  pairing of the characters on $\sigma_1,\sigma_2$ when we restrict $\E$ to $C$. Note that this well-defined: If, for instance, $u_{11}=u_{12}$, then it seems we could reorder and glue different sheets together, however then the corresponding $\sigma_{11}$ and $\sigma_{12}$ will already be identified via $\sim$, so we get the same object.

The piecewise linear function $\Psi_\E$  on $\bE$ is obtained from the local versions above. There is also a projection map $\pi: \bE \to |\Sigma|$ given by mapping each $\sigma_{ik}$ isomorphically onto $\sigma_i$. By a line segment in $\bE$, we will mean a connected subset $l$ such that the projection $\pi(l)$ is a line segment in $|\Sigma|$ and so that the fiber over each point of $\pi(l)$ is a single point.  We will say that $\Psi_\E$ is concave, if it is concave when restricted to any line segment.

Fix a line segment in $\bE$ with endpoints $v$ and $w$.  The projection $\pi(w)$ has to be contained in at least one maximal cone $\sigma$. There are $r$ sheets in $\bE$, each mapping isomorphically to $\sigma$. The sheets are by construction in bijection to characters $u_{i}$ such that $\E|_{U_\sigma} \simeq \oplus_{i=1}^r \OO(\ddiv(u_{i}))$. Thus  which sheet $w$ is contained in corresponds to a distinguished character $u_{i}$. We say that any such $u_i$ is a character obtained by moving in a straight line from $v$ inside $\bE$.

\begin{proposition} \label{proposition:branchedcovercriterion}
Given a toric vector bundle $\E$ on a smooth complete toric variety $X_\Sigma$, the following are equivalent.
\begin{enumerate}
    \item $\E$ is nef (ample)
    \item $\Psi_\E$ is (strictly) concave for all $i$ 
    \item $\Psi_\E(v) \geq  \langle u_{i},v \rangle$ ($\Psi_\E(v) >  \langle u_{i},v \rangle$) where $u_{i}$ is any character obtained by moving in a straight line from $v$ inside $\bE$.
\end{enumerate}
\end{proposition}
\begin{proof}
Note that (2) and (3) only say something about $\Psi_\E$ restricted to line segments, thus we fix a line segment $l$ in $\bE$. The projection $\pi(l) \subset |\Sigma|$ is a line segment in $N_\Q$. Let $\Sigma'$ be the subfan of $\Sigma$ consisting of all cones in $\Sigma$ intersecting $\pi(l)$, as well as all their faces. The line segment $l$ determines a character $u_\sigma$ on each maximal cone $\sigma \in \Sigma'$. The collection of these characters define a divisor $D_l$ on $X_{\Sigma'}$. By construction we see that  $\Psi_\E|_l = \phi_{D_l} \circ \pi|_l$. We claim that $\E$ is nef (ample) if and only if $\phi_{D_l}|_l$ is (strictly) concave for all $l$. Assuming this claim we see that the proposition follows from the corresponding statement for divisors.

The proof of the claim depends on an adaption of standard techniques on toric line bundles. We put the detailed statement in  \cref{lemma:localconcavity} below. To prove the claim note $\E$ is nef  if and only if $\E|_C$ is nef for any invariant curve $C$. By \cite[Corollary 5.5]{HMP} $\E|_C$ splits as a sum of line bundles $\OO_C(D_i)$, thus nefness is equivalent to the inequalities $D_i \cdot C \geq 0$, for all $i,C$. Let $C$ be given by the cone $\tau$ of codimension one, and assume $\tau$ is contained in the maximal cones $\sigma$ and $\sigma'$. By \cite[Proposition 6.3.8]{CLS}
\[D_i \cdot C =  \langle u_{\sigma'}-u_{\sigma},v \rangle = \phi_d(v) - \langle u_{\sigma},v \rangle , \]
where $v$ is the generator of $\tau^\perp$ which is positive on $\sigma'$ and $D_i$ corresponds to the characters $u_\sigma,u_{\sigma'}$. 

By  the above inequality we see that if $\E$ is nef then clause $(4)$ of  \cref{lemma:localconcavity}  is satisfied, thus by $(1)$ the support function $\phi_D|_l$ is concave for any $l$.

Conversely if the support function is concave for all lines $l$, then we can pick $l$ starting in the generator of $\tau^\perp$ which is positive on $\sigma'$ and ending in the interior of $\sigma$. Then the inequality $(2)$ of \cref{lemma:localconcavity} is exactly stating that the corresponding divisor satisfies $D_i \cdot C \geq 0$. Since this is true for all divisors and curves, $\E$ is nef.

The case of ampleness is the same argument,  only with concavity replaced with strict concavity, as well as that we require inequalities to be strict.
 \end{proof}
 
 \begin{lemma} [{cf. \cite[Lemma 6.1.5]{CLS}}] \label{lemma:localconcavity}
 Let $D$ be a Cartier divisor on a toric variety corresponding to the fan $\Sigma$. Fix a line $l$ which is  contained in the support of $\Sigma$. Then the following are equivalent:
 \begin{enumerate}
     \item The support function $\phi_D|_l : l \to \R$ is concave.
     \item $\phi_D(v) \geq \langle u_\sigma,v \rangle $ for all $v \in l$ and maximal cones $\sigma$ intersecting $l$.
     \item $\phi_D(v)= \max \langle u_\sigma,v \rangle$, where $v \in l$ and the maximum is over maximal cones $\sigma$ intersecting $l$.
     \item For every cone of codimension one $\tau = \sigma \cap \sigma'$ such that $\sigma$ and $\sigma'$ intersects $l$ there is $v_0 \in \sigma' \setminus \sigma$ with $\phi_D(v_0) \geq \langle u_\sigma,v_0 \rangle$.
 \end{enumerate} 
 \end{lemma}
 \begin{proof}
 The proof is essentially identical to the proof of \cite[Lemma 6.1.5]{CLS}, except that one has to replace all instances of the support of $\Sigma$ with $l$ and only consider cones which intersects $l$.
 \end{proof}
 
 \begin{remark}
 In \cite{KavehManon} Kaveh and Manon prove that   a toric vector bundle is nef (ample) if and only if $\E$ is what they call buildingwise (strictly) convex. This statement is similar to the above, although formulated in a different language.
 \end{remark}

\section{Vector bundles with non-globally generated symmetric powers} \label{counterexample}

We here describe a sequence of ample bundles $\E_k$ such that $S^k \E_k$ is not globally generated. This example was worked out jointly with Greg Smith (private correspondence).

Let $X_\Sigma$ be the smooth complete toric surface with rays 
\[ v_1= \begin{bmatrix} 1 \\ 0 \end{bmatrix},v_2=\begin{bmatrix} 1 \\ 1 \end{bmatrix}, v_3=\begin{bmatrix} 1 \\ 2 \end{bmatrix},
 v_4=\begin{bmatrix} 0 \\ 1 \end{bmatrix},v_5=\begin{bmatrix} -1 \\ 0 \end{bmatrix},v_6=\begin{bmatrix} 0 \\ -1 \end{bmatrix}.   \]
Let $\E$ be a rank $2$ toric vector bundle on $X_\Sigma$. Assume that for each ray, the filtration has both a one-dimensional and a two-dimensional vector space appearing. Let the integers for which the filtration jumps on ray $v_i$ be given by $a_i,b_i$, where $b_i>a_i$. Letting $p_i$ be the one-dimensional space appearing in ray $v_i$, we assume that $q:=p_1=p_2=p_3=p_4$ and that $q,p_5,p_6$ are pairwise distinct.

We have that $\E$ is ample if and only if the restriction to any invariant curve is ample. 
This is equivalent to the list of inequalities below. These can be derived as follows. An invariant curve $C$ corresponds to a ray of the fan. Restricted to the curve $\E|_C$ splits as a sum of line bundles $D_1$ and $D_2$. The numbers below are $C \cdot D_i$, for all $i,C$. Fix for instance the curve $C$ corresponding to $v_2$. The splitting type is determined by the restriction of $\E$ to the maximal cones $\sigma_{12} = \Cone(v_1,v_2)$ and  $\sigma_{23}=\Cone(v_2,v_3)$ containing $v_2$. We have that $\sigma_{12}$ corresponds to the characters $(b_1,b_2-b_1)$ and $(a_1,a_2-a_1)$ and $\sigma_{23}$ corresponds to $(2b_2-b_3,b_3-b_2)$  and $(2a_2-a_3,a_3-a_2)$. By \cite[Corollary 5.10]{HMP} the restriction to $C$ is determined by a unique pairing of the characters from $\sigma_{12}$ and $\sigma_{23}$, such that the differences of the characters are parallel to $v_2^\perp$. Doing this we obtain the first two inequalities in the list, the ten others correspond to the other five curves.
\begin{align*}
b_1+b_3-2b_2 &> 0 &
a_1+a_3-2a_2 &> 0 \\
b_2+b_4-b_3 &> 0 &
a_2+a_4-a_3 &> 0 \\
b_3+a_5-2b_4 &> 0 &
a_3+b_5-2a_4 &> 0 \\
b_4+b_6 &> 0 &
a_4+a_6 &> 0 \\
b_5+b_1 &> 0 &
a_5+a_1 &> 0 \\
b_6+a_2-a_1 &> 0 &
a_6+b_2-b_1 &> 0.
\end{align*}

Let the bundle $\E_k, k \geq 1$ be defined by
\begin{align*}
 a_1&=1 & b_1&=4 \\ 
 a_2&=-2 & b_2&=6k-1 \\ 
 a_3&=3k-6 & b_3&=12k-5 \\ 
 a_4&=6k-5 & b_4&=6k-3 \\
 a_5&=0 & b_5&=9k-3 \\
 a_6&=6-6k & b_6&=4. \\
\end{align*}
We can check that the ampleness inequalities are satisfied for $\E_k$. However we can show that $S^k \E_k$ is not globally generated:

The characters on the cone $\sigma_{2,3}$ are given by $u_1=(2a_2-a_3,a_3-a_2),u_2=(2b_2-b_3,b_3-b_2)$. 
On $\sigma_{23}$ we need to choose a matroid vector corresponding to the character $u=u_1+(k-1)u_2$. This will be of the form $q^{k-1}p_5$ or $q^{k-1}p_6$. In the first case we need to have
\[ k a_6 \geq a_2-a_3 + (k-1)(b_2-b_3) \]
while in the second case we need to have
\[ k a_5 \geq a_3-2a_2 + (k-1)(b_3-2b_2) \]
Inserting the chosen values and cleaning up we see that these inequalities are
\[ 0 \geq k \]
\[ 0 \geq 1 .\]
Both of these are clearly false, hence $S^k \E_k$ cannot be globally generated.

In conclusion, this example shows that on a toric variety $X$ there cannot exist a number $k$ such that for any ample toric vector bundle $\E$, we have that $S^k \E$ is always globally generated, not even for rank two bundles on a toric surface.

\section{Pullbacks under multiplication maps} \label{posSection:multiplication}
For a toric variety $X_\Sigma$ there is, for any positive integer $k$, a toric morphism $f_k: X_\Sigma \to X_\Sigma$ called  ``multiplication by $k$'' or  the ``toric Frobenius map''.  It is given by the map of lattices $N \to N$ which multiplies any element by $k$. For a toric vector bundle $\E$ on $X_\Sigma$ we thus have bundles $f_k^\ast \E$, for any positive integer $k$. 

The interest in the multiplication maps come from the fact that for a line bundle $\Li$, the pullback $f_k^\ast \Li$ is ``more positive'' than $\Li$, since  $f_k^\ast \Li= \Li^{\otimes k}$. This is similar to the Frobenius map for varieties in characteristic $p$, which was used by Deligne-Illusie-Reynard to prove the Kodaira vanishing theorem \cite{DeligneIllusie}. The toric Frobenius has been used to great effect in proving vanishing theorems on toric varieties \cite{Fujino}, \cite[Chapter 9]{CLS}. Thus, one would expect that applying $f_k^\ast$ to a toric vector bundle $\E$ will only ``increase the positivity'' of $\E$. However, here we show that such analogous statements are not true for several positivity properties of toric vector bundles.

We now fix $k$ and set $\F = f_k^\ast \E$. By \cite[Proposition 6.2.7]{CLS} the pullback $f_k^\ast D$ is just $kD$, multiplication by $k$. Since $\E$ locally is a sum of divisors this shows that the Klyachko filtrations of $\F$ is given by $E^i(j) =F^i(jk)$. Thus the vector spaces in the filtrations are the same, in particular $M(\E) = M(\F)$. Also for $v \in M(\F)$ we have that the associated divisor $D_v = kE_v$, where $E_v$ is the divisor associated to $v \in M(\E)$.

\begin{lemma}
For any $k,l$ we have that $S^l f_k^\ast \E \simeq f_k^\ast S^l \E$.
\end{lemma}
\begin{proof}
By the above we know the Klyachko filtrations of $f_k^\ast \E$. We also know the Klyachko filtrations of a symmetric power. Writing out the filtrations for any ray, we see that they are identical.
\end{proof}

\begin{proposition} \label{theorem:multiplication} 
For any positive integer $k$ we have that $f_k^\ast \E$ is globally generated, nef, big or ample   if and only if $\E$ is globally generated, nef, big or ample, respectively.
\end{proposition}
\begin{proof}
For nef (resp. ample) the argument is easy: $\E$ is nef (resp. ample) if and only if for each T-invariant curve $C$, $\E|_C$ is nef (resp. ample). Now $\E|_C$ is a sum of line bundles $D_i$ and $f_k^\ast \E|_C$ is the sum of $kD_i$. Since $D_i$ is nef (ample) if and only if $kD_i$ is nef (resp. ample), the result follows.

The statement on bigness follows directly from  \cref{theorem:bignessCriteria}. Since the polytopes for $f_k^\ast \E$ are simply $k$ times the polytopes for $\E$ we see that  a polytope in  the parliament of one of the bundles being full-dimensional is equivalent to the corresponding polytope for the other bundle being full-dimensional.

For global generation we use the criterion in Theorem \ref{theorem:gg}. $\E$ is globally generated if and only if for any maximal cone $\sigma$, with $\E|_{U_\sigma} \simeq \oplus_{i=1}^r \OO(\ddiv(u_i))$ there exists a basis $v_1,\ldots,v_r \in M(\E)$ such that $u_i \in P_{v_i}$. Similarly $f_k^\ast \E$ is globally generated if and only if there exists a basis $w_1,\ldots,w_r \in M(f_k^\ast \E)$ such that $ku_i \in P_{w_i}$. Now since $D_{w_i} = kD_{v_i}$, $P_{w_i} = kP_{v_i}$ and $M(\E) = M(f_k^\ast \E)$ we see that the two criteria above imply each other.

\end{proof}
\begin{remark}
The end of the above proof shows that the corresponding equivalence can fail for some other notions of positivity: in particular \cite[Theorem 6.2]{DJS} implies that if $\E$ separates $1$-jets then $f_k^\ast \E$ separates $k$-jets, but clearly the converse is not true. This is because separating $k$-jets correspond to certain edges having lattice length at least $k$. Thus any toric vector bundle $\E$ where one of the relevant edges has length $1$ does not separate $k$-jets for any $k \geq 2$, however $f_k^\ast \E$ will separate $k$-jets.
\end{remark}

\begin{example} \label{example:noggbound}
Fix an ample toric vector bundle $\E$ on $X_\Sigma$ which  is not globally generated or very ample, for instance the examples in \cref{counterexample}. Then $\E|_C$ is a vector bundle $\sum \OO(a_i)$ on $C \simeq \p^1$ for any invariant curve $C \subset X_\Sigma$ with $a_i >0$. Now $\F=f_k^\ast \E$ will satisfy $\F|_C = \sum \OO(ka_i)$. In particular, the restriction to any invariant curve is greater than or equal to $k$, which can be picked arbitrarily large. However by Theorem \ref{theorem:multiplication}, $\F$ is still not globally generated or very ample. Thus there cannot in general exist any bound depending on dimension, or even the fan,  guaranteeing that if each summand in $\E|_C$ is more positive than this bound, then $\E$ is globally generated or very ample. For line bundles this exists, depending only on dimension, by the statements implying Fujita's conjecture \cite{PayneFujita}.
\end{example}

The above examples suggests that the multiplication maps might be less useful for studying toric vector bundles, compared to their usefulness for line bundles: Many of the vanishing theorems in toric geometry follow from the fact that for a line bundle one obtains injections $H^i(X_\Sigma,\OO(D)) \subset H^i(X_\Sigma,\OO(f_k^\ast D))=H^i(X_\Sigma,\OO(k D))$ \cite{Fujino}. If $D$ is ample then $f_k^\ast D$ is very ample for large $k$. For toric vector bundles, the analogous statement does not hold, thus it is not clear if one can use  this technique to get vanishing theorems for positive toric vector bundles.

\begin{example} \label{example:highercoh}
Let $\E = T_{\p^n}(-1)$ and let $\F = f_{n+1}^\ast \E$. From the Euler sequence one obtains the two exact sequences
\[ 0 \to \OO(-1) \to \OO^{n+1} \to \E \to 0, \]
\[ 0 \to \OO(-n-1) \to \OO^{n+1} \to \F \to 0. \]
The second sequence is obtained by pulling back the Euler exact sequence along $f_{n+1}^\ast$; pullback is exact for vector bundles. 
We see that all higher cohomology of $\E$ vanishes, however the same is not the case for $\F = f_{n+1}^\ast \E$.

Letting $\G = (f_{n+1}^\ast T_{\p^n})(-n)$, we have the exact sequence
\[ 0 \to \OO(-n) \to \OO(1)^{n+1} \to \G \to 0. \]
In particular $\G$ is very ample and $H^i(\p^n,\G) = 0$ for $i>0$. However, the higher cohomology of $f_k^\ast \G$ is nonzero for $k \geq 2$. We note that we also have that $f_k^\ast \G$, $k \geq 2$ is a very ample toric vector bundle such that all polytopes in the parliament are very ample, but with non-vanishing higher cohomology.
\end{example}

\begin{remark}
Let $\E$ be a toric vector bundle and let $\F = f_k^\ast \E$. Then, since restricted to $U_\sigma$ the bundle splits as a direct sum, we see that $c_i(\F) = k^i c_i(\E)$. This implies that for any weighted degree $n$ polynomial $P$ in the Chern classes, where $c_i$ has weight $i$, we have that $P(c_1(\F),c_2(\F),\ldots,c_r(\F)) = k^n P(c_1(\E),c_2(\E),\ldots,c_r(\E))$. Then \cref{example:highercoh} implies that there cannot exist any homogeneous polynomial in the Chern classes of $\E$ such that if this polynomial is larger than some constant, all higher cohomology of very ample bundles vanishes. 
\end{remark}

\bibliography{ref}
\bibliographystyle{alpha}
\end{document}